\documentclass{amsart}
\usepackage{amsmath}
\usepackage{amsfonts}
\usepackage{amssymb}
\usepackage{amsthm}
\usepackage{url}
\usepackage{dsfont}
\usepackage{enumitem}
\usepackage{graphicx}
\usepackage{caption}
\usepackage{subcaption}
\usepackage{comment} 
\usepackage[noadjust]{cite}
\usepackage{stmaryrd}
\usepackage{todonotes}
\usepackage{hyperref}
\numberwithin{equation}{section}

\newtheorem{prop}{Proposition}[section]
\newtheorem{lem}[prop]{Lemma}
\newtheorem{defn}[prop]{Definition}
\newtheorem{thm}[prop]{Theorem}

\newtheorem{rem}[prop]{Remark}

\newtheorem{conj}[prop]{Conjecture}

\DeclareMathOperator{\Bl}{Bl}
\newcommand{\vext}{\mathfrak{X}_{S(\omega)}}
\newcommand{\ddb}{i\partial \bar\partial}

\newcommand{\BlX}{\widetilde{X}}
\newcommand{\Blball}{\widetilde{B}_{\varepsilon}^j}
\newcommand{\ball}{B^j_{\varepsilon}}

\title{Blowing up extremal Poincar\'e type manifolds}
\author{Lars Martin Sektnan}
\address{D\'epartement de math\'ematiques\\
  Universit\'e du Qu\'ebec \`{a} Montr\'eal\\
  Case postale 8888, succursale centre-ville \\
  Montr\'eal (Qu\'ebec)\\
  H3C 3P8 \\
  Canada}
\email{lars.sektnan@cirget.ca}

\begin{document}
\maketitle
\begin{abstract} We prove a version of the Arezzo-Pacard-Singer blow-up theorem in the setting of Poincar\'e type metrics. We apply this to give new examples of extremal Poincar\'e type metrics. A key feature is an additional obstruction which has no analogue in the compact case. This condition is conjecturally related to ensuring the metrics remain of Poincar\'e type.
\end{abstract}

\setcounter{tocdepth}{2}
\tableofcontents

\section{Introduction}

An important problem in K\"ahler geometry is to find whether or not an extremal K\"ahler metric exists in a given K\"ahler class. The Yau-Tian-Donaldson conjecture states that this should related to some notion of algebro-geometric stability.

A natural question is what happens in the unstable case, when no extremal metric exists on some compact polarised manifold $(V,L)$. In \cite{Don02}, \cite{donaldson05} and \cite{donaldson09}, Donaldson suggested the following conjectural picture. From the algebro-geometric point of view there should be an optimal destabiliser (see \cite{szekelyhidi08} for results in this direction in the toric case). If we for simplicity assume that the central fibre of the optimal destabiliser is a union of smooth irreducible components $X_i$ meeting in smooth divisors $D_i$, then the conjectural picture is that when taking a minimising sequence for the Calabi functional, suitable rescalings yield complete extremal metrics on $X_i \setminus D_i$. 

A large class of extremal metrics on $X \setminus D$ for a smooth, compact $X$ and smooth divisor $D$, come from metrics with cusp/Poincar\'e type singularities, whose study goes back at least as far as \cite{carlsongriffiths72}. For example, there are the K\"ahler-Einstein metrics of \cite{chengyau80}, \cite{rkobayashi84} and \cite{tianyau87} (see also \cite{schumacher98} and \cite{rochonzhang12} for further results on the asymptotics of these metrics), toric metrics on $\mathbb{P}^n \setminus \mathbb{P}^{n-1}$ of \cite{abreu01} and \cite{bryant01}, and metrics on the total space of certain $\mathbb{P}^1$-bundles using the momentum construction, see \cite{Szethesis}.
This type of asymptotics forms a natural candidate for the asymptotics of the extremal metrics conjectured to arise when the original compact manifold $V$ of the previous paragraph is unstable. 

In the present article, we will prove a perturbation result for extremal metrics on $X \setminus D$ with Poincar\'e type singularities, which yield further examples from the given known ones. However, our results will also show that there are obstructions to such perturbations novel to the Poincar\'e type case. The obstructions come from holomorphic vector fields on $D$.

There are already examples of complete extremal metrics on some $X \setminus D$ which are not of Poincar\'e type (and where no extremal Poincar\'e type metric can exist). See \cite{Szethesis} for examples which only occur at the boundary of the extremal cone, and \cite{apostolovauvraysektnan17} for examples which fill the whole K\"ahler cone. This leads to a YTD type conjecture for the existence of these type of metrics, first formulated by Sz\'ekelyhidi in \cite{Szethesis}, and further refined in \cite{apostolovauvraysektnan17}.
\begin{conj}[\cite{Szethesis}, \cite{apostolovauvraysektnan17}]\label{ptconj} Let $(X,L,D)$ be a triple consisting of a line bundle $L$ over a compact K\"ahler manifold $X$ with a smooth divisor $D$. Then there exists an extremal Poincar\'e type metric on $X \setminus D$ in the class $c_1 (L)$ if and only if 
\begin{itemize} \item the relative Donaldson-Futaki invariant of any test configuration for the triple $(X,L,D)$ is non-negative, with equality if and only if the test configuration is a product;
\item the Sz\'ekelyhidi numerical constraints is satisfied;
\item the restriction of the extremal vector field of $X$ for the class $c_1 (L)$ to $D$ equals the extremal vector field of $D$ for the class $c_1(L)_{|D}$.
\end{itemize}
\end{conj}

The new obstruction to perturbations in the Poincar\'e type case that we find in this article can be thought of as giving further evidence to this conjecture, as it is the exactly the final condition in the above conjecture which is the novel condition required to obtain extremal Poincar\'e type metrics on the blow-up. On the other hand, the results suggests that one might expect more types of asymptotics than the Poincar\'e type asymptotics to occur for the complete metrics conjectured to arise when a compact K\"ahler manifold is unstable.  

\begin{rem} For the third point in Conjecture \ref{ptconj} to make sense, we need to know that there is an associated extremal vector field, as in the compact case of \cite{futakimabuchi95}.  That this holds follows from the results of Section \ref{analyseweights} and in particular Lemma \ref{extvfistangent}, combined with the argument of Lemma \ref{projectionchange}. This implies that there is a unique projection to the holomorphic vector fields on $X$ that are tangent to $D$, once a maximal torus has been chosen and we work with metrics invariant under this maximal torus.
\end{rem}

\subsection{Statement of results}

The main theorem is an extension of the blow-up theorems of Arezzo-Pacard \cite{arezzopacard06,arezzopacard09}, Arezzo-Pacard-Singer \cite{arezzopacardsinger11} and Sz\'ekelyhidi \cite{szekelyhidi12, szekelyhidi15b} to the Poincar\'e type case. Below we will let $G$ be the connected component of the identity of the group of automorphisms of $X$ that preserve, but not necessarily fix $D$. Let $T$ be a (compact) torus in $G$, chosen such that it contains the extremal vector field $\vext$ of the Poincar\'e type metric, and let $K$ be a maximal compact subgroup of $G$ containing $T$. Let $H$ be the centraliser of $T$ in $K$. We will let  $\mathfrak{t}, \mathfrak{k}$ and $\mathfrak{h}$ denote the Lie algebras of $T,K$ and $H$, respectively. 

We will let $\mu : X \setminus D \to \mathfrak{k}$ denote the normalised moment map, where we are using the natural inner-product to consider this as a map to $\mathfrak{k}$ instead of $\mathfrak{k}^*$. Note that if $p$ is a fixed point of $T$, then $\mu (p) \in \mathfrak{h}$.

Next, let $\mathfrak{X}_D$ denote the extremal vector field on $D$ for the class $\Omega_{|D}$ relative to a maximal torus of the automorphism group of $D$ which contains the automorphisms of $D$ coming from automorphisms in $T$. Finally we will let $\mathfrak{X}_{\varepsilon}$ denote the extremal vector field for the class $\Omega_{\varepsilon}$, defined by \eqref{epsilonclass} below, on the complement of $D$ in the blow-up, relative to the lift of the maximal torus $T$ to the blow-up. This vector field is a vector field on $X$ preserving $D$, and so can be restricted to $D$. 

\begin{thm}\label{mainblthm} Let $X$ be a compact complex manifold and suppose $\omega \in \Omega$ is an extremal Poincar\'e type K\"ahler metric on the complement of a smooth divisor $D$. Suppose further that $\omega$ is invariant under the action of the maximal compact subgroup $K$ in $G$, containing $\vext$, the extremal vector field of $\omega$. 

Let $p_1, \cdots, p_k \in X \setminus D$ and $a_1, \cdots, a_k > 0$ be chosen such that $\vext(p_i) = 0$ for all $i$ and 
\begin{align}\label{condbala} \sum_i a_i^{n-1} \mu (p_i) \in \mathfrak{t}.
\end{align}
Suppose also that  
\begin{align}\label{condgene} \mathfrak{t} +  \langle  \mu (p_1), \cdots, \mu(p_k) \rangle_{\mathbb{R}}  = \mathfrak{h}
\end{align} 
and that any vector field in $\mathfrak{h}$ vanishing at all the points $p_i$ necessarily is in $\mathfrak{t}$. Finally, suppose that 
\begin{align}\label{extremalvf} \mathfrak{X}_{\varepsilon| D } = \mathfrak{X}_D. 
\end{align} 
Then there is a constant $\varepsilon_0 >0$ such that for all $\varepsilon \in (0, \varepsilon_0)$ the blow-up of $X$ at the points $p_i$ admits an extremal Poincar\'e type K\"ahler metric with Poincar\'e type singularity along $\pi^{-1} (D) \cong D$ in the class
\begin{align}\label{epsilonclass} \Omega_{\varepsilon} = \pi^*(\Omega) - \varepsilon^2 \big(\sum_i a_i [E_i] \big),
\end{align}
where $\pi$ is the blow-down map and $E_i = \pi^{-1} (p_i)$ are the exceptional divisors. 
\end{thm}
Note that in the statement above, the conditions \eqref{condbala} and \eqref{condgene} are exactly analogous to the conditions in the compact case, whereas the condition \eqref{extremalvf} is novel to the Poincar\'e type setting. This extra condition arises from additional cokernel elements associated to the Lichnerowicz operator of a Poincar\'e type metric.

\subsection{Strategy for proving the main theorem}\label{strategysect} Good knowledge of the linear theory for the Lichnerowicz operator associated to a Poincar\'e type metric is crucial to prove our main theorem, Theorem \ref{mainblthm}. We will now describe the general setup and strategy for proving the theorem once this linear theory is in place. We take the same approach as in \cite{arezzopacardsinger11} for the compact case. 

Even though it is important to consider all points being blown up simultaneously in the construction, let us for simplicity consider the case when we are blowing up one point $p$ in $X$. The basic strategy is to 
\begin{itemize} \item consider the blow-up $\Bl_p X $ as being made up of two parts -- the complement $X \setminus B_{\varepsilon}$ of a ball in $X$, and $\widetilde{B_{\varepsilon}}$, the pre-image of a ball about the origin in $\Bl_0 \mathbb{C}^n$ ;
\item construct many extremal metrics on $X \setminus B_{\varepsilon}$ and $\widetilde{B_{\varepsilon}}$; 
\item show that under the assumptions of Theorem \ref{mainblthm} we can match two of these extremal metrics on each component along their common boundary to construct an extremal metric on the whole of $\Bl_p X$.
\end{itemize}

The constructions of the extremal metrics on $\widetilde{B_{\varepsilon}}$ is identical to the construction of \cite{arezzopacardsinger11} in the compact case. The new step is constructing extremal metrics of Poincar\'e type on the complement of a ball about the blown up point in $X \setminus D$, starting from the given one on the whole of $X \setminus D$. A new technical point for the Poincar\'e type case is that we will need to allow some metrics which also are not necessarily extremal in order to achieving the matching. We then return to consider whether or not the metrics actually are extremal at the very end, once we have a metric on the blow-up. There is a finite dimensional set of obstructions, which is precisely the condition \eqref{extremalvf} in Theorem \ref{mainblthm}.

\subsection{Organisation of the article} In Section \ref{background}, we recall some background on Poincar\'e type metrics. In particular, we discuss more precisely the asymptotics we assume. We introduce the function spaces that will be important for us and discuss a useful decomposition of these spaces, using a tubular neighbourhood discussion as in \cite{auvray13}. This will be important when we in Section \ref{fredholmpropssect} prove Theorem \ref{ptfredholmthm} which shows that except for a discrete set of weights, the Lichnerowicz operator associated to a Poincar\'e type metric is always a Fredholm operator on the weighted function spaces. 

In Section \ref{analyseweights} we explicitly find the kernel and cokernel of the Lichnerowicz operator for the weights relevant to proving Theorem \ref{mainblthm}. The key is the characterisation in Theorem \ref{ptfredholmthm} of the cokernel for a given weight in terms of the kernel for a different weight, together with an integration by parts lemma and a construction due to Auvray. In Section \ref{lintheory} we extend the linear theory of Sections \ref{fredholmpropssect} and \ref{analyseweights} to the \textit{doubly weighted spaces} that will be needed in the analysis.

Section \ref{nonlinearsect} is devoted to proving the blow-up theorem. It carries out the non-linear analysis needed to prove Theorem \ref{mainblthm}. This follows the strategy of Arezzo-Pacard closely as outlined in subsection \ref{strategysect} above, and is split into several steps -- getting better and better approximations to the extremal equation, before then showing an extremal metric can be found. The only fundamentally new step here compared to the compact case is to use the extra assumption \eqref{extremalvf} to ensure that the metrics we obtain in the end actually are extremal.

We end the article in Section \ref{egsect} by showing how the main theorem gives rise to new examples of compact smooth K\"ahler manifolds $X$ with a smooth divisor $D$ admitting an extremal Poincar\'e type metric on $X \setminus D$ in a K\"ahler class $\Omega$.

\subsection*{Acknowledgements:} This work was begun as a part of the author's PhD thesis at Imperial College London. I would like to thank my supervisor Simon Donaldson for his encouragement, support and insight. I gratefully acknowledge the support from the Simons Center for Geometry and Physics, Stony Brook University at which some of the research for this paper was performed. I am thankful to Hugues Auvray and G\'abor Sz\'ekelyhidi for inviting me to fruitful trips to Orsay and Notre Dame, respectively, which were very beneficial for completing this work. I also thank Ruadha\'i Dervan for helpful comments and discussions. Finally, I would like to thank my PhD examiners Mark Haskins and Michael Singer for many valuable comments on the version which was part of the author's PhD thesis. This work was supported by funding from the CIRGET.

\section{Background on Poincar\'e type metrics}\label{background}

\subsection{Definition of a metric of Poincar\'e type}
Consider the punctured unit (open) disk $B_1^* \subseteq \mathbb{C}$ with the metric
\begin{align}\label{localcusp} \frac{|dz|^2}{(|z|  \log |z|)^2}.
\end{align}
Here we use the notation $|dz|^2 = dx^2 + dy^2$, where $z = x + iy$. This is the standard cusp or Poincar\'e type metric on $B_1^*$. Note that if one lets $t = \log (- \log(|z|))$ and $\theta$ be the usual angular coordinate, this equals
\begin{align*}  \frac{|dz|^2}{(|z|  \log |z|)^2} = dt^2 + e^{-2t} d \theta^2.
\end{align*}
A computation shows that a K\"ahler potential for this metric is $4 \log (- \log ( |z|^2) )$.

In \cite{auvray17}, Auvray made a definition of a compact complex manifold admitting a metric with the asymptotics of the product of this metric with a smooth metric on $D$, near $D$. His definition was for a simple normal crossings divisor. However, in this article we will only be considering a smooth irreducible divisor, so we only recall the notion of such a metric in this context.

Given such a divisor $D$, one can define a model potential for a Poincar\'e type metric as follows. Pick a holomorphic section $\sigma \in H^0(X,\mathcal{O} ( D))$ such that $D$ is the zero set of $\sigma$. Fix a Hermitian metric $|\cdot|$ on $\mathcal{O} (D)$, which we assume satisfies $ | \sigma | \leq e^{-1}$. Let $\omega_0$ be a K\"ahler metric on the whole of the compact manifold $X$ and for a constant $\lambda>0$, let $f=\log( \lambda - \log(|\sigma|^2) )$. For sufficiently large $\lambda$, $\omega_{f} = \omega_0 - i \partial \overline{\partial} f$ is then a positive $(1,1)$-form on $X \setminus D$. Poincar\'e type metrics are defined to be metrics on $X \setminus D$ defined by a potential with similar asymptotics to $f$ near $D$.
\begin{defn}[{\cite[Def. 0.1]{auvray17},\cite[Def. 1.1]{auvray13}}]\label{poincaretypemetdefn} Let $X$ be a compact complex manifold and let $D$ be a smooth irreducible divisor in $X$. Let $\omega_0$ be a K\"ahler metric on $X$ in a class $\Omega \in H^2(X, \mathbb{R})$. A smooth, closed, real $(1,1)$ form on $X \setminus D$ is a Poincar\'e type K\"ahler metric if 

{
\itemize

\item $\omega$ is quasi-isometric to $\omega_{f}$. That is, there exists a $C$ such that
\begin{align*} C^{-1} \omega_{f} \leq \omega \leq C \omega_{f}.
\end{align*} 
}

Moreover, the class of $\omega$ is $\Omega$ if
\itemize
\item $\omega = \omega_0 + i \partial \overline{\partial} \varphi$ for a smooth function $\varphi$ on $X \setminus D$ with $| \nabla_{\omega_f}^j \varphi |$ bounded for all $j \geq 1$ and $\varphi = O(f)$.
\end{defn}

If $\omega$ is a Poincar\'e type metric in a class $\Omega$, then $\omega$ has finite volume. In fact, its volume equals that of any smooth metric in $\Omega$ on the whole of $X$. 

\subsection{Function spaces}\label{fnspaces}

We begin by defining the Sobolev spaces $W^{2,k} ( X \setminus D) = W^{2,k} ( X \setminus D, \omega)$ associated to a Poincar\'e type metric $\omega$ on $X \setminus D$. First, $L^2 (X \setminus D)$ is defined to be the completion of $C^{\infty}_c (X \setminus D)$ with respect to the norm
\begin{align*} \|  f \|_{L^2(X\setminus D)} = \big( \int_{X \setminus D}| f|^2 \omega^n \big)^{\frac{1}{2}}.
\end{align*}
For a function of class $C^k$, we then let $\| \nabla^k f \|_{L^2(X\setminus D)} $ denote the $L^2$-norm of $| \nabla^k f|$, where the $\nabla^k f$ denotes the higher covariant derivatives of $f$ and the pointwise norms are computed with respect to the norm induced by the metric $g$ on the tensor bundle in which $\nabla^k f$ lies. The $W^{2,k}$-Sobolev norm is then defined by 
\begin{align*} \|  f \|_{W^{2,k}(X\setminus D)} = \sum_{i=0}^k \| \nabla^k f \|_{L^2(X\setminus D)}.
\end{align*}
We let $W^{2,k}(X\setminus D)$ be the completion of $C^{\infty}_c (X \setminus D)$ under this norm. 

Next we define the H\"older spaces. These are defined used \textit{quasi-coordinates}, introduced in \cite{chengyau80}. Quasi-coordinates are immersions into $X \setminus D$ depending on a parameter $\varsigma \in (0,1)$ whose union covers $U \setminus D$ for some neigbourhood $U$ of $D$ in $X$. In the model case $B_1^* \times D$, we can cover the product of $D$ with a smaller punctured disk $B_{r}^* \times D$ by the following maps.

For $\varsigma \in (0,1)$ and some fixed $R \in (\frac{1}{2},1)$, let $\phi_{\varsigma} : B_R (0) \rightarrow B_1^* $ be given by 
\begin{align*} z \mapsto e^{\frac{1+\varsigma}{1-\varsigma} \frac{1+z}{1-z}} .
\end{align*}
As $\varsigma$ varies between $0$ and $1$ this covers a punctured ball $B_r^*$ around 0 in $B_1^*$ (see e.g. \cite[Sec. 2]{rkobayashi84}). Let $\omega_{PT}$ be the K\"ahler form associated to the standard cusp metric on $B_1^*$. One then has 
\begin{align*} \phi^*_{\varsigma} (\omega_{PT}) = \frac{i dz \wedge d \overline{z}}{(1-|z|^2)^2}
\end{align*}
which is quasi-isometric to the Euclidean metric independently of $\varsigma$. The $C^{k,\alpha}$-norm for a function $f$ on $B_1^*$ is then defined to be 
\begin{align*} \| f \|_{C^{k,\alpha} (B_1 \setminus B_{\frac{r}{2}})} +  \textnormal{sup}_{\varsigma \in (0,1)} \big( \| \phi_{\varsigma}^* f\|_{C^{k,\alpha}_{B_R (0)}} \big).
\end{align*}

In the global case in arbitrary dimension, one first fixes a finite set of charts $U_1, \cdots, U_d$ for $X$ such that the union of $U_i \cap D$ covers $D$ and $U_i \cap D = \{ z \in U_i : z_1 =0 \}$, i.e. $D$ is in $U_i$ given by the vanishing of the first coordinate function. By composing the coordinate map with the product of $\phi_{\varsigma}$ and the identity map on the last $n-1$ coordinates, this gives maps 
\begin{align*} \Phi_{\varsigma}^i : B_R(0) \times V_i \rightarrow X \setminus D
\end{align*}
for some open sets $V_i \subseteq \mathbb{C}^{n-1}$, whose union over all $\varsigma \in (0,1)$ and $i \in \{1, \cdots d \}$ covers $U \setminus D$ for some open set $U \subseteq X$ containing $D$. Letting $U_0$ be an open set with compact closure in $X\setminus D$ and which contains the complement of $U$, the $C^{k,\alpha}$-norm on $X \setminus D$ is then defined to be
\begin{align}\label{ckanormdefn} \| f \|_{C^{k,\alpha} (X \setminus D) } = \| f \|_{C^{k,\alpha}(U_0)} + \textnormal{max}_{i=1}^d \textnormal{sup}_{\varsigma \in (0,1)} \| (\Phi_{\varsigma}^i)^* f \|_{C^{k,\alpha} (B_R(0) \times V_i)}.
\end{align}

To obtain Fredholm properties of the relevant operators for the scalar curvature problems, we now add weights to our discussion, see \cite[Defn. 3.1]{auvray17}. 
\begin{defn} Let $\eta \in \mathbb{R}$. The $L^2$-norm with weight $\eta$ on $X \setminus D$ is
\begin{align}\label{globall2defn} \| f \|^2_{L_{\eta}^2(X \setminus D)} = \int_{X \setminus D} |f|^2 e^{-2 \eta t} \omega^n ,
\end{align}
where we recall that $\omega$ is our Poincar\'e type metric on $X\setminus D$. We define the weighted $W^{2,k}$-norm by
\begin{align}\label{globalw2kdefn} \| f \|^2_{W_{\eta}^{2,k}(X \setminus D)} = \sum_{i=0}^k \| \nabla^i f \|_{L^2_{\eta} (X \setminus D)}.
\end{align}
Finally, we define the $C^{k, \alpha}_{\eta}$-space to be 
\begin{align*} C^{k,\alpha}_{\eta} (X \setminus D) = e^{\eta t} C^{k,\alpha} (X \setminus D)
\end{align*}
equipped with the norm
\begin{align*} \| f \|_{C^{k,\alpha}_{\eta} (X \setminus D)} = \| e^{- \eta t} f\|_{C^{k,\alpha} (X \setminus D)}.
\end{align*}
\end{defn}

We will end the section by recalling a useful decomposition of functions on $X\setminus D$ that will be important for the estimates we will prove later. This relies on a tubular neighbourhood discussion of Auvray as in \cite{auvray17} and \cite[Sec. 3]{auvray13}. 

The exponential map obtained from a smooth metric $\omega_0$ defined on the whole of $X$ identifies a neighbourhood $\mathcal{V}$ of $D$ in the normal holomorphic bundle of $D$ with an initial tubular neighbourhood $\mathcal{N}$ of $D$ in $X$. The holomorphic normal bundle admits an $S^1$-action and so, by possibly reducing $\mathcal{V}$ to ensure it is preserved by the $S^1$-action, we can then endow $\mathcal{N}$ with an $S^1$-action, too. From the projection to $D$ in the holomorphic normal bundle of $D$, we similarly get a projection 
\begin{align}\label{normalproj1} \pi : \mathcal{N} \setminus D \rightarrow D ,
\end{align} which moreover is invariant under the $S^1$-action.

The function $\mathfrak{u} = \log(-\log(|\sigma|^2))$ can be perturbed to a function $t : X \setminus D \rightarrow \mathbb{R}$ which is $S^1$-invariant in $\mathcal{N} \setminus D$ and such that to any order, it is $O(e^{-\mathfrak{u}})$. Introducing a parameter $A$, we can then define $\mathcal{N}_A$ to be the union of $D$ with the points $x \in \mathcal{N} \setminus D$ for which $t(x) \geq A$. We will take $A$ to be fixed and write $\mathcal{N} = \mathcal{N}_A$, i.e. we chose $\mathcal{N}$ to be $\mathcal{N}_A$ from the beginning. One then obtains a map 
\begin{align}\label{normalproj2} \Pi = (\pi, t) : \mathcal{N} \setminus D \rightarrow D \times [A, \infty),
\end{align} 
which is an $S^1$-fibration.

Auvray further constructed a $1$-form $\vartheta$ associated to the $S^1$-action above. This has the key properties that in trivialising charts for $\mathcal{N}$ where $D$ is given by $z_1 = 0$, to any order it satisfies
\begin{align*} \vartheta = d \theta + O(1),
\end{align*}
 where $\theta$ is the the angular coordinate associated to $z_1$. $\vartheta$ also integrates to $2\pi$ on each fibre of the $S^1$-fibration \eqref{normalproj2}. The model Poincar\'e type metric $\tilde{\omega}$ in a class $[\omega_0]$ then has an expansion 
\begin{align*} \tilde{g} = dt^2 + e^{-2t} \vartheta^2 + \pi^* h_0 + O(e^{-t}),
\end{align*}
at any order. Here $h_0$ is the metric on $D$ associated to the K\"ahler form $\omega_0$ on $X$ restriced to $D$.

Still following \cite[Sec. 3]{auvray13}, we can use the above to decompose a function $f$ on $X \setminus D$ orthogonally into 
\begin{align}\label{fndecomp} f = f_0 + f^{\perp}, 
\end{align}
where $f_0$ is supported in $\mathcal{N}$ and is $S^1$-invariant, and where $f^{\perp}$ has average $0$ on each fibre of $\ref{normalproj2}$ near $D$. Thus in the tubular neighbourhood about $D$ we can then identify $f_0$ with a map $[A, \infty ) \times D \rightarrow \mathbb{R}$. 

We then have that $f \in C^{k,\alpha}_{\textnormal{loc}} (X \setminus D)$  is in $C^{k,\alpha} (X \setminus D)$ if and only if each of $f_0$ and $f^{\perp}$ are. Moreover, using the identification of $f_0$ with a function on the cylinder, $f_0 \in C^{k,\alpha} (X \setminus D)$ if and only if 
\begin{align*} f_0 \in C^{k,\alpha}([A, \infty) \times D),
\end{align*}
the cylindrical H\"older space. The $C^{k,\alpha}$-norm of function $\psi$ on the cylinder $[A, \infty) \times D$ is defined as
\begin{align*} \textnormal{sup}_{s \geq A+1} \| \psi \|_{C^{k,\alpha} ([s-1,s+1] \times D, dt^2 + g_D)}
\end{align*}
where $t$ is the coordinate on $[A,\infty)$ and $g_D$ is some metric on $D$. This equivalence between the cylindrical and Poincar\'e type H\"older norms for $S^1$-invariant functions is proved e.g. in \cite[Lemma 6.7]{sektnanthesis}.

\subsection{Basic properties}In this section we collect a couple of basic properties of the function spaces, that we will call upon later.

\begin{lem}\label{holderinsobolev} Let $\eta, \eta' \in \mathbb{R}$. Then $$C^{k,\alpha}_{\eta} \big( X \setminus D \big) \subseteq W^{2,k}_{\eta'}\big( X \setminus D \big)$$ if and only if $\eta < \eta' + \frac{1}{2}.$
\end{lem}
\begin{proof} Note that $e^{\eta t} \in C^{k,\alpha}_{\eta}$ and has finite $L^2_{\eta'}$ norm if and only if $2 \eta - 2 \eta' - 1 < 0$,  i.e. if and only if $\eta' > \eta - \frac{1}{2}$, showing one direction of the claim. Here we used that, in charts, the volume form of the Poincar\'e type metric is mutually bounded with the volume form $e^{-t} dt \wedge d\theta \wedge \omega_D$.

Conversely, for the case $\eta = 0$, this follows because for all $f \in C^{k,\alpha}(X \setminus D)$, $| \nabla^i f | $ is bounded and $X \setminus D$ has finite volume with respect to the volume form given by $e^{- 2 \eta' t} \omega^n$, where $\omega$ is a Poincar\'e type metric, if $\eta' > - \frac{1}{2}$. For the case of other values of $\eta$,  note that because $e^{\eta t} \in C^{k,\alpha}_{\eta}(X \setminus D)$ for all $k, \alpha$, we have that $|\nabla^i f| \in C^{k-i, \alpha}_{\eta} (X \setminus D)$ for all $i$. It then follows that $| e^{-\eta t} \nabla^i f|$ is bounded for every $i$, and so we can apply the previous argument.
\end{proof}

\begin{lem}\label{normcomparisonw2klem} For all $\delta, \varepsilon > 0$ there exists a compact subset $K \subseteq X \setminus D$ and $C>0$ such that 
\begin{align}\label{normcomparisonw2k} \| f \|_{W^{2,k+4}_{\eta+\delta}} \leq \varepsilon \| f \|_{W^{2,k+4}_{\eta}}  + C \| f \|_{L^2(K)} .
\end{align} 
\end{lem}
\begin{proof}
For a real number $s$, let $K_s = \{ x : t(x) \leq s \}$. Note that since they are compact subsets, on each $K_s$ all the different weighted norms are equivalent, with the constant of equivalency depending on $s, \delta$ and $\eta$. In particular, 
\begin{align}\label{equivalenceKs} \| f \|_{W^{2,k+4}_{\eta+\delta}(K_s)} \leq  C \|  f \|_{L^{2}(K_s)}
\end{align}
for some constant $C$ depending on the same parameters.

Note that $\| f \|_{W^{2,k+4}_{\eta+ 2 \delta}} = \| e^{-\frac{\delta}{2} t} f \|_{W^{2,k+4}_{\eta}}$. We then have that for any $s$, 
$$\| f \|_{W^{2,k+4}_{\eta+\delta}(X \setminus D)} = \| e^{-\frac{\delta}{2} t} f \|_{W^{2,k+4}_{\eta} (K_s^c)} + \| f \|_{W^{2,k+4}_{\eta+\delta}(K_s)}   .$$
Pick $s$ such that $e^{-\frac{\delta}{2} s} < \varepsilon $. Then as $\delta $ is positive, $e^{-\frac{\delta}{2} t} < \varepsilon$ on $K_s^c$ and so the above combined with equation \eqref{equivalenceKs} gives that 
\begin{align*}\| f \|_{W^{2,k+4}_{\eta+\delta}(X \setminus D)} < \varepsilon \| e^{\delta t} f \|_{W^{2,k+4}_{\eta} (K_s^c)} + C \| f \|_{L^2(K_s)} 
\end{align*}
from which \eqref{normcomparisonw2k} follows by picking $K = K_s$.
\end{proof}

\section{Fredholm properties of the Lichnerowicz operator}\label{fredholmpropssect}

The goal of this section is to prove the following theorem on the Fredholm properties of the Lichnerowicz operator, under a stronger assumption on the asymptotics of a Poincar\'e type metric.
\begin{thm}\label{ptfredholmthm} Let $X$ be a K\"ahler manifold, let $D$ be a smooth irreducible divisor in $X$ and suppose $\omega$ is a Poincar\'e type metric on $X \setminus D$ that satisfies equation \eqref{metricasymptotics}. Suppose $\eta$ is not an indicial root for the Lichnerowicz operator $\mathcal{D}^* \mathcal{D} = \mathcal{D}^*_{\omega} \mathcal{D}_{\omega}$. Then $\mathcal{D}^* \mathcal{D}$ is a Fredholm operator $C^{k+4,\alpha}_{\eta} \rightarrow C^{k,\alpha}_{\eta}$. Moreover, we have that
\begin{align}\label{imageintermsofker} \textnormal{Im } \mathcal{D}^* \mathcal{D}_{C^{k+4, \alpha}_{\eta}} = ( \textnormal{Ker} (\mathcal{D}^*\mathcal{D}_{C^{k+4,\alpha}_{1 - \eta}} ) )^{\perp},
\end{align}
where $\perp$ denotes the $L^2$-inner product and subscripts denote the domains of the operators.
\end{thm}
The set of indicial roots will be defined below. It is a discrete subset of $\mathbb{R}$.

\subsection{Assumption on the metric} We first describe the asymptotics we will assume for the rest of the article. Recall from the previous section that the model Poincar\'e type metric $\tilde{\omega}$ in a class $[\omega_0]$ has an expansion 
\begin{align*} \tilde{g} = dt^2 + e^{-2t} \vartheta^2 + \pi^* h_0 + O(e^{-t}),
\end{align*}
where $t$ is a function invariant under the $S^1$-action on a tubular neighbourhood of $D$ and asymptotic to $\log(-\log(|\sigma|^2)$. We will consider the case when the metric $g$ associated to $\omega$ satisfies
\begin{align}\label{metricasymptotics} g = a(dt^2 + e^{-2t} \vartheta^2) + \pi^* g_D + O(e^{- \eta t}),
\end{align}
for some K\"ahler metric $g_D$ on $D$ and $a, \eta>0$, again up to any given order. Crucially for us, Auvray has in \cite[Thm. 4.8]{auvray14} shown that when $D$ is smooth, i.e. has no intersecting irreducible components, this expansion holds for extremal Poincar\'e type metrics, where $h$ is in fact an extremal metric on $D$. In particular, this holds in the case we are considering in this article. We need the parameter $a$, as this is changed when changing the K\"ahler form to $\omega + k \ddb \big( \log ( - \log ( | \sigma |^2 ) ) \big)$ for some constant $k$. 

\subsection{The key estimates} The goal of this section is to prove two key estimates for the proof of Theorem \ref{ptfredholmthm}, namely Propositions \ref{globalw2kestimate} and \ref{ptholderreg}. 

We start by proving the inequality \eqref{globallichineq}. We will adopt a strategy similar to that of Biquard in \cite{biquard97} for the Laplace operator, and use the decomposition of a function $f$ into an $S^1$-invariant part and an orthogonal part near $D$. The Lichnerowicz operator respects this decomposition. The reason for doing this is that then we can apply the theory of Lockhart-McOwen in \cite{lockhartmcowen85} to the $S^1$-invariant part of the function. We now recall the parts of the Lockhart-McOwen theory that will be relevant for us. 

The results of Lockhart-McOwen (which build on the results of \cite{agmonnirenberg63} and others) are in the setting of elliptic operators on manifolds with cylindrical ends. They apply in particular to elliptic operators $L$ of order $l$ on the model cylinder $C_Y = (0,\infty) \times Y$, where $Y$ is compact, that are translation invariant in the cylinder coordinate $t$. For such an operator,
\begin{align*}  \| f \|_{W^{2,k+l}_{\delta} (C_Y)} \leq c \| Lf \|_{W^{2,k}_{\delta} (C_Y)} 
\end{align*}
for all $\delta$ which are not an \textit{indicial root} of $L$. An indicial root is a $\delta \in \mathbb{R}$ such that there is a solution to the eigenvalue problem of the Fourier transform of $L$ of the form $$e^{i \delta t}p(t,x),$$ where, for each $x \in Y$, $p$ is a polynomial in $t$. The set of indicial roots is a discrete subset of $\mathbb{R}$. 

Using this inequality together with a partition of unity argument, Lockhart-McOwen then obtain that, for these weights, on the half-cylinder $H_Y = [1, \infty) \times Y$, we have that for every $b>1$, there is $c$ such that 
\begin{align}\label{lmineq}  \| f \|_{W^{2,k+l}_{\delta} (H_Y)} \leq c \big(\| Lf \|_{W^{2,k}_{\delta} (H_Y)} + \| f \|_{L^2([1,b]\times Y)} \big) .
\end{align}

For us, the relevant model operator is 
\begin{align}\label{modellich} f &\mapsto \frac{1}{2}\big(  \frac{\partial^2 }{\partial t^2} - \frac{\partial }{\partial t} \big)^2 (f) -  \big( \frac{\partial^2 }{\partial t^2} - \frac{\partial }{\partial t} \big) (\Delta_D f)  - \frac{1}{2} \big( \frac{\partial^2 }{\partial t^2} - \frac{\partial }{\partial t} \big) (f)  + \mathcal{D}^*_D \mathcal{D}_D f,
\end{align}
which is the Lichnerowicz operator associated to the standard cusp metric on $B_1^* \times D$ acting on $S^1$-invariant functions, where the $S^1$-action is the product of the standard action on $B_1^*$ and the trivial action on $D$.  This model operator corresponds to the case when $a=1$ in the assumption \eqref{metricasymptotics}. We will focus on this case, but the argument goes over to all other positive values of $a$. 

We let operators with a subscript $D$ denote operators on $D$ defined with respect to the metric on $D$ from equation \eqref{metricasymptotics}. In particular, note that this operator has coefficients that are translation invariant in $t$, and so the Lockhart-McOwen theory applies to this operator.

\begin{prop}\label{globalw2kestimate} Let $X$ be a compact K\"ahler manifold and let $D \subseteq X$ be a smooth divisor. Suppose that $\omega$ is a Poincar\'e type K\"ahler metric satisfying \eqref{metricasymptotics}. Suppose $\eta$ is not an indicial root for the operator \eqref{modellich}. Then there exists a compact subset $K \subseteq X \setminus D$ and a $c > 0$ such that
\begin{align}\label{globallichineq} \| f \|_{W^{2,k}_{\eta} (X \setminus D)} \leq c( \|\mathcal{D}^*_{\omega} \mathcal{D}_{\omega} f\|_{W^{2,k-4}_{\eta} (X \setminus D)} + \| f \|_{L^2 (K)}).
\end{align}
\end{prop}
\begin{proof} If $K' \subset K$ are compact subsets of $X \setminus D$ such that $K'$ is contained in the interior of $K$, then a similar inequality holds with $X \setminus D$ replaced by $K'$. Therefore we can restrict to considering functions supported in the tubular neighbourhood $\mathcal{N}$ about the divisor. Here we will use the decomposition in equation \eqref{fndecomp} of $f$ into an $S^1$-invariant part $f_0$ and a complementary part $f^{\perp}$.

For the $S^1$-invariant part $f_0$ of the function, we can then identify $f_0$ with a function on the cylinder $[A, \infty) \times D$. The Lichnerowicz operator $\mathcal{D}^* \mathcal{D}$ differs from the Lichnerowicz operator of the model cusp metric acting $S^1$-invariant functions on $B^* \times D$ by $O(e^{-t})$, due to the assumption \eqref{metricasymptotics}. Here $B$ is the ball of radius $e^{- e^{A}/2}$ about the origin in $\mathbb{C}$, the radius value corresponding to $t=A$. Lemma \ref{normcomparisonw2klem} implies that it suffices to prove that \eqref{globallichineq} holds for functions on $[A, \infty) \times D$ and the operator in equation \eqref{modellich}. But this is a constant coefficients operator in $t$ and so the Lockhart-McOwen theory applies. The inequality \eqref{lmineq} is then exactly what we require for the bound on $f_0$.

For the complementary part $f^{\perp}$, recall that its average on each fibre of the fibration $\Pi$ in equation \eqref{normalproj2} is $0$. This implies that there is a $c'>0$ such that 
\begin{align}\label{Lminussqlapl} \| \big( \Delta^2 - L \big) f^{\perp} \|_{W^{2,k}_{\eta}} \leq c' \| f^{\perp} \|_{W^{2,k+3}_{\eta +1 }}.
\end{align} 
For the two operators $\Delta^2$ and $L$ agree to leading order, and so is a bounded order three  operator. But in \cite[p.9]{auvray13} showed that for the component $f^{\perp}$ one gets higher decay than $\eta$ for the lower order derivatives. For example for the first derivative, this follows simply by integrating over a fibre, and using that the derivative have to vanish somewhere, since the mean is null on the fibre. Thus we get the estimate \eqref{Lminussqlapl}.  

To finish the proof, note that Biquard (\cite[Theorem 5.1 and Lemma 6.3]{biquard97}) showed a similar bound to the one we require for the Laplace operator. Thus it also holds for the square of the Laplacian, i.e. there exists a $C>0$ and compact subset $K \subseteq X \setminus D$ such that 
\begin{align*} \| f \|_{W^{2,k+4}_{\eta}} \leq C \big(\| \Delta^2 f \|_{W^{2,k}_{\eta}} + \| f \|_{L^2(K)} \big) .
\end{align*} 
Combining this with the estimate \eqref{Lminussqlapl}, we get that, by possibly increasing $C$, 
\begin{align*} \| f^{\perp} \|_{W^{2,k+4}_{\eta}} \leq C \big(\| L f^{\perp} \|_{W^{2,k}_{\eta}} + \| f^{\perp} \|_{W^{2,k+4}_{\eta +1 }} + \| f^{\perp} \|_{L^2(K)} \big) .
\end{align*}

By Lemma \ref{normcomparisonw2klem}, we can pick a compact subset $K' \subseteq X \setminus D$ and constant $C'>0$ such that 
\begin{align*} \| f^{\perp} \|_{W^{2,k+3}_{\eta+1}} &\leq \frac{1}{2C} \| f^{\perp} \|_{W^{2,k+3}_{\eta}}  + C' \| f^{\perp} \|_{L^2(K')} \\
&\leq \frac{1}{2C} \| f^{\perp} \|_{W^{2,k+4}_{\eta}}  + C' \| f^{\perp} \|_{L^2(K')}  .
\end{align*} 
Combining this with the above, we obtain the required estimate by possibly increasing $K$ and $C$.
\end{proof}

Next we will prove a regularity result. This follows \cite[Lem. 12.1.1]{pacardnotes}

\begin{prop}\label{ptholderreg} Suppose $f \in L^{2}_{\eta - \frac{1}{2}}$ and suppose that $\mathcal{D}^*_{\omega} \mathcal{D}_{\omega} f \in C^{k-4, \alpha}_{\eta}$ in the sense of distributions for a weight $\eta$. Then $f \in C^{k, \alpha}_{\eta}$. Moreover, there is a $c >0$ such that
\begin{align*} \| f \|_{C^{k+4, \alpha}_{\eta}} \leq c \big( \| \mathcal{D}_{\omega}^* \mathcal{D}_{\omega} f\|_{C^{k,\alpha}_{\eta}} + \| f \|_{L^{2}_{\eta - \frac{1}{2}}} \big).
\end{align*}
\end{prop}

\begin{proof} From the usual elliptic theory, it follows that $f \in C^{k, \alpha}_{\textnormal{loc} }(X \setminus D)$, so we need to estimate the $C^{k,\alpha}_{\delta}$-norm. Using the Schauder estimates and that the weighted norms are equivalent to the unweighted norm on any compact subset $K$ of $X \setminus D$, we get immediately that there is a $c> 0$, depending on $K$, such that
\begin{align*}  \| f \|_{C^{k+4, \alpha}_{\eta}(K)}  \leq c \big( \| \mathcal{D}^* \mathcal{D} f\|_{C^{k,\alpha}_{\eta} (X \setminus D)} + \| f \|_{L^{2}_{\eta - \frac{1}{2}} (X \setminus D)} \big).
\end{align*}

Thus we have to show that the required bound holds near the divisor. As before, we divide the argument into one for the $S^1$-invariant part and one for the component $f^0$.

We begin with the case when $f$ is $S^1$-invariant with respect to the local $S^1$-action, and so can be identified with a function on a cylinder $[\lambda , + \infty ) \times D$ for some fixed $\lambda$. There is a $c >0$, independent of $s$, such that for all $s>\lambda +2$, we have
\begin{align*} \| f \|_{C^{k+4, \alpha} ([s-1, s+1] \times D )} \leq c( \| \mathcal{D}^* \mathcal{D} f\|_{ C^{k, \alpha}([s-2, s+2] \times D ) } + \| f \|_{L^{2} ([s-2,s+2] \times D)} ).
\end{align*}
Next, we multiply this by $e^{-\eta s}$. We have 
\begin{align*}\| e^{-\delta s }f \|^2_{L^{2} ([s-2,s+2] \times D)} 
&\leq \textnormal{max} \{ e^{-4 \delta}, e^{4 \delta} \} \| e^{- \delta t} f \|_{L^2 ([s-2,s+2] \times D) },
\end{align*}
since $e^{- 2 \delta t} \leq e^{- 2 \delta (s+2)}$ in $[s-2,s+2] \times D$ if $\delta <0$ and $e^{- 2 \delta t} \leq e^{- 2 \delta (s-2)}$ if $\delta >0$. Here the $L^2$-norm is computed with respect to the volume form $dt \wedge \omega_{D}^{n-1}$ and $\omega_{D}$ is the smooth metric $\omega_0$ restricted to $D$. Since the volume form of the Poincar\'e type metric $\omega$ is mutually bounded with $e^{-t} dt \wedge d\theta \wedge \omega_{D}^{n-1}$ and $f$ is $S^1$-invariant, it follows that 
\begin{align*}
\| e^{-\delta s }f \|^2_{L^{2} ([s-2,s+2] \times D)}& \leq \frac{\textnormal{max} \{ e^{-4 \delta}, e^{4 \delta} \}}{2\pi} \| f \|^2_{L^2_{\delta-\frac{1}{2}} (X \setminus D) }.
\end{align*}
 So by possibly increasing $c$ we get an inequality of the form
\begin{align*} \| e^{- \delta s} f \|_{C^{k, \alpha} ([s-1, s+1] \times D )} \leq c( \| e^{- \delta s} \mathcal{D}^* \mathcal{D} f\|_{ C^{k-4, \alpha}([s-2, s+2] \times D ) } + \| f \|_{L^{2}_{\delta - \frac{1}{2}} (X \setminus D)}).
\end{align*} 

Now, by a similar argument as above, one can show that $\| e^{- \delta s} f \|_{C^{k, \alpha} ([s-1, s+1] \times D )} $ is mutually bounded with $\| e^{- \delta t} f \|_{C^{k, \alpha} ([s-1, s+1] \times D )} $, independently of $s$. Similarily for $\| e^{- \delta s} \mathcal{D}^* \mathcal{D} f\|_{ C^{k-4, \alpha}([s-2, s+2] \times D ) }$. Thus there is a $c > 0$ such that for all $s > \lambda + 2$, we have
\begin{align*} \| e^{- \delta t} f \|_{C^{k, \alpha} ([s-1, s+1] \times D )} \leq c( \| e^{- \delta t} \mathcal{D}^* \mathcal{D} f\|_{ C^{k-4, \alpha}([s-2, s+2] \times D ) } + \| f \|_{L^{2}_{\delta  - \frac{1}{2}} (X \setminus D)}).
\end{align*} 
Thus by taking the supremum over all $s > \lambda + 2$, we get the required result.

For the remaining component $f^{\perp}$ the proof reduces, as in the proof of Proposition \ref{globalw2kestimate}, to the case of the Laplacian. Indeed, by the same argument, we get an estimate of the form
\begin{align*} \| f^{\perp} \|_{C^{k+4, \alpha}_{\eta}} \leq c \big( \| \mathcal{D}_{\omega}^* \mathcal{D}_{\omega} f^{\perp} \|_{C^{k,\alpha}_{\eta}} + \| f^{\perp} \|_{L^{2}(K)} + \| f^{\perp} \|_{L^{2}_{\eta - \frac{1}{2}}} \big)
\end{align*}
for some compact subset $K \subseteq X \setminus D$. But since the $(\eta - \frac{1}{2})$-weighted norm is equivalent to the unweighted norm on any compact subset of $X \setminus D$, the term $\| f^{\perp} \|_{L^{2}(K)}$ can be bounded by a constant multiple of $\| f^{\perp} \|_{L^{2}_{\eta - \frac{1}{2}}}$, which yields the required inequality.
\end{proof}

\subsection{Proof of Theorem \ref{ptfredholmthm}}\label{ptfredholmthmpf} We will now use the estimates of the previous section to prove Theorem \ref{ptfredholmthm}. The finite dimensionality of the kernel and the closedness of the image of $\mathcal{D}^* \mathcal{D}$ follows directly from Proposition \ref{globalw2kestimate} using standard contradiction arguments, see e.g. \cite[Ch. 9]{pacardnotes}. Note that due to the inclusion $W^{2,k}_{\eta} \subseteq W^{2,k}_{\eta'}$ whenever $\eta \leq \eta'$, the finite-dimensionality of the kernel of the Lichnerowicz operator holds for all weights, not just away from the indicial roots. The finite dimensionality of the cokernel (and hence the Fredholm property of the Lichnerowicz operator), will then follow from the characterisation in equation \eqref{imageintermsofker} which we prove below.

To show equation \eqref{imageintermsofker}, we use the regularity result, Proposition \ref{ptholderreg}. We first establish that for any weight $\delta$, the kernel of the adjoint of $\mathcal{D}^* \mathcal{D}$ on $(L^2_{\delta})^*$ can be identified with the kernel of $C^{4,\alpha}_{\frac{1}{2}-\delta}$. 

First note that $L^2_{\delta}(X\setminus D)^*$ can be identified with $L^2_{-\delta}$ by using the $L^2$-inner product. Also, the operator $\mathcal{D}^* \mathcal{D}$ is formally self-adjoint, hence if $f \in L^2_{-\delta}$ is in the kernel of the adjoint operator $ (L^2_{\delta})^* \rightarrow (W^{2,2}_{\delta})^*$, it solves $\mathcal{D}^* \mathcal{D} (f) = 0$ in the sense of distributions. By Proposition \ref{ptholderreg} and the fact that by Lemma \ref{holderinsobolev}, $C^{4,\alpha}_{\frac{1}{2} - \delta - \varepsilon} \subseteq L^2_{-\delta}$ if $\varepsilon > 0$, it follows that $f \in C^{4,\alpha}_{\frac{1}{2}-\delta -  \varepsilon}$ for any $\varepsilon > 0$.

Next, let $L_{\varepsilon}$ denote the operator $\mathcal{D}^* \mathcal{D}$ with domain $W^{2,4}_{\eta - \frac{1}{2}+ \varepsilon}$. Note that
\begin{align*} \textnormal{Im } L_{\varepsilon} = \big( \textnormal{Ker} (L_{\varepsilon}^*) \big)^*.
\end{align*}
The above paragraph applied to $\delta = \eta  - \frac{1}{2} + \varepsilon$ gives that $\textnormal{Ker} (L_{\varepsilon}^*)$ can be identified with $\textnormal{Ker } L_{C^{k+4}_{1- \eta - 2 \varepsilon}}$. Hence 
\begin{align*} \big( \textnormal{Ker} (L_{\varepsilon}^*) \big)^* = \big( \textnormal{Ker } L_{C^{k+4}_{1- \eta - 2 \varepsilon}} \big)^{\perp}.
\end{align*}
However, since $\eta$ is not an indicial root, $\textnormal{Ker } L_{C^{k+4}_{1- \eta - 2 \varepsilon}}$ does not change for $|\varepsilon |$ sufficiently small and so in particular
\begin{align*} \big( \textnormal{Ker} (L_{\varepsilon}^*) \big)^* = \big( \textnormal{Ker } L_{C^{k+4}_{1- \eta }} \big)^{\perp}.
\end{align*}

To complete the proof, note that again by Lemma \ref{holderinsobolev}, $C^{k,\alpha}_{\eta} \subseteq L^2_{\eta - \frac{1}{2} + \varepsilon}$ if and only if $\varepsilon > 0$. Hence by Proposition \ref{ptholderreg} and picking $\varepsilon > 0$ sufficently small in the above
\begin{align*} \textnormal{Im } L_{C^{k+4,\alpha}_{\eta}} &= \textnormal{Im }L_{\varepsilon} \cap C^{k,\alpha}_{\eta} \\
&= \big( \textnormal{Ker } L_{C^{k+4}_{1- \eta }} \big)^{\perp}.
\end{align*}
This completes the proof of Theorem \ref{ptfredholmthm}.

\section{Explicit analysis of the (co)-kernel for the relevant weights}\label{analyseweights}

It is only when $\eta < 0$ that all the elements of $C^{k,\alpha}_{\eta}$ can be potentials for a Poincar\'e type metric. We would therefore like to tell what the kernel and cokernel of the Lichnerowicz operator is for small negative weights. The first goal of this section is to prove such a characterisation. We will then make a slight adjustment to these spaces which are more suited for the perturbation problem that we want to solve.

\subsection{The standard spaces} 
We begin by describing the main result of this subsection, Proposition \ref{kercokerlem}. 

We will need the following definition. In the statement $\mathfrak{h}$ is the space of real holomorphic vector fields on $X$.
\begin{defn}[{\cite[Defn. 1.1]{auvray14}} ]\label{parallelvfsdefn} Let $\mathfrak{X} $ be a real holomorphic vector field on $X$. We say $\mathfrak{X}$ is \textnormal{tangent to $D$} if, when writing $\mathfrak{X}$ as the real part of $\sum_i f_i \frac{\partial}{\partial z_i}$ in coordinates near $D$ where $D$ is given by $z_1 = 0$, we have $f_1 = 0$ on $D$. We write $\mathfrak{h}_{//}^D $ for the subspace of $\mathfrak{h}$ of such vector fields. The space of potentials for vector fields in $\mathfrak{h}_{//}^D$ with zeros is denoted $\overline{\mathfrak{h}_{//}^D} $. 
\end{defn}

Vector fields tangent to $D$ as in Definition \ref{parallelvfsdefn} enter in our discussion of weighted spaces because of the following result of Auvray.
\begin{lem}[{\cite[Lem. 5.2]{auvray17} }]\label{parallelvfslem} Let $\mathfrak{X} \in \mathfrak{h}$ be a real holomorphic vector field on $X$. Then its $L^2$-norm $\| \mathfrak{X} \|_{L^2 (X \setminus D, \omega)}$ with respect to a Poincar\'e type K\"ahler metric $\omega$ on $X \setminus D$ is finite if and only if $\mathfrak{X} \in \mathfrak{h}_{//}^D$.
\end{lem} 

We will also need to extend certain functions on $D$ to $X$. We achieve this as follows. For a function $f$ on $D$, we can extend $f$ to an $S^1$-invariant function $\pi^* f$ near $D$ by using the projection $\pi : \mathcal{N} \setminus D \rightarrow D$ from equation \eqref{normalproj1}. Using an $S^1$-invariant bump function supported on $\mathcal{N}$ and only depending on $t$, we can then consider $\chi \pi^*f$ as a globally defined function. 

Now, let $\mathfrak{h}^{D}$ be the space of real holomorphic vector fields on $D$ and let $s$ be the codimension of the subspace consisting of vector fields induced by a vector field tangent to $D$. Recall that $\varphi$ was the K\"ahler potential of the Poincar\'e type metric $\omega$. The functions on $D$ we need to extend to a neighbourhood of $D$ are given by taking a function $f \in \textnormal{Ker }  \mathcal{D}^*_D \mathcal{D}_D$, and then pulling back the corresponding real holomorphic vector field $\mathfrak{X}_f$ to the tubular neighbourhood around $D$. Here $\mathcal{D}^*_D \mathcal{D}_D$ denotes the Lichnerowicz operator on $D$ associated to the metric $\omega_{0|D}$. We then let $\psi_f$ be given by 
\begin{align}\label{psif} \psi_f = \chi (\pi^* f + d \varphi (\pi^* \mathfrak{X}_f) ).
\end{align}
Given a basis $1=f_0, f_1, \cdots , f_r$ for $ \textnormal{Ker }  \mathcal{D}^*_D \mathcal{D}_D$, we let $\psi_i = \psi_{f_i}$. We will assume $f_{0}, \cdots, f_{s}$ form a basis for the subspace of $\textnormal{Ker }  \mathcal{D}_D^* \mathcal{D}_D$ of potentials for vector fields on $D$ induced by vector fields on $X$.

Finally, let $\mathfrak{h}_0^{D}$ denote the vector fields in $\mathfrak{h}_{//}^{D}$ whose induced vector field on $D$ vanishes, and let $\overline{\mathfrak{h}_0^{D}}$ denote the potentials for such vector fields. Our characterisation of the kernel and cokernel for the relevant weights is then the following. 

\begin{prop}\label{kercokerlem} Consider the Lichnerowicz operator
\begin{align*} \mathcal{D}^* \mathcal{D} = \mathcal{D}^*_{\omega} \mathcal{D}_{\omega} : C^{4,\alpha}_{\eta} (X \setminus D) \rightarrow C^{0,\alpha}_{\eta} (X \setminus D)
\end{align*}
on the Poincar\'e type weighted spaces. Then there is a $\kappa > 0$ such that 
\begin{align*}
 \textnormal{Ker } \mathcal{D}^* \mathcal{D}_{C^{4,\alpha}_{\eta}} =& \overline{\mathfrak{h}_{//}^D}  \textnormal{ if } \eta \in (0,1), \\ 
\textnormal{Ker } \mathcal{D}^* \mathcal{D}_{C^{4,\alpha}_{\eta}} \subseteq& \overline{\mathfrak{h}_0^{D}} \textnormal{ and is of codimension }1 \textnormal{ if } \eta \in (- \kappa, 0), \\
 C^{0,\alpha}_{\eta} \cap   \overline{\mathfrak{h}_{//}^D}^{\perp} =& \textnormal{ Im} (\mathcal{D}^* \mathcal{D}_{C^{4,\alpha}_{\eta}}) \oplus \langle \mathcal{D}^* \mathcal{D} ( \psi_i ): i \in \{s+1, \cdots, r \} \rangle   \textnormal{ if } \eta \in (-\kappa,0).
\end{align*}
\end{prop}
Note that by Theorem \ref{ptfredholmthm}, the cokernel for $\eta \in (0,1)$ can be identified with the kernel for the weight $1-\eta$, which also lies in $(0,1)$ and hence equals $\overline{\mathfrak{h}^D_{//}}$. Proposition \ref{kercokerlem} then says that when going to small negative weights, the kernel decreases to a codimension one subspace of $\overline{\mathfrak{h}_0^{D}}$ and the cokernel increases by the span of the elements $\mathcal{D}^* \mathcal{D} ( \psi_i )$ for $ i \in \{s+1, \cdots, r \}$.

The following integration by parts result will be used several times.
\begin{lem}\label{ptintbyparts} Let $f \in C^{4,\alpha}_{\eta}$ and $g \in C^{4,\alpha}_{\eta'}$ with $\eta + \eta' < 1$. Then 
\begin{align*} \int_{X \setminus D} \mathcal{D}^* \mathcal{D} (f) g \omega^n = \int_{X\setminus D} \langle \mathcal{D}(f) , \mathcal{D} (g)\rangle \omega^n.
\end{align*}
\end{lem}

\begin{proof} For simplicity we consider the case when $\eta' = \eta$, so in particular $\eta < \frac{1}{2}$. It will however be clear that the same argument goes through for any choice of $\eta$ and $\eta'$ satisfying $\eta + \eta' < 1$.

Let $\chi : \mathbb{R} \rightarrow \mathbb{R}$ be a bump function supported on $(- \infty, 1]$ and equal to $1$ in $(- \infty, 0]$ and let $\chi_a (x) = \chi (x -a)$. We can consider $\chi$ as a function on $X\setminus D$ by composing with the function $t$. We then have that
\begin{align}\label{lim1} \lim_{a \rightarrow \infty} \int_{X \setminus D}  \chi_a \mathcal{D}^* \mathcal{D} (f) g \omega^n &= \int_{X \setminus D} \mathcal{D}^* \mathcal{D} (f) g \omega^n , \\
\label{lim2} \lim_{a \rightarrow \infty}\int_{X\setminus D} \chi_a \langle \mathcal{D}(f), \mathcal{D} (g) \rangle  &= \int_{X\setminus D} \langle \mathcal{D}(f) , \mathcal{D} (g) \rangle \omega^n.
\end{align}

Since $\chi_a g$ has compact support, it follows that
\begin{align*}\int_{X \setminus D}  \chi_a \mathcal{D}^* \mathcal{D} (f) g \omega^n = \int_{X \setminus D}   \langle \mathcal{D} (f) , \mathcal{D} (\chi_a g) \rangle \omega^n .
\end{align*}
This differs from 
\begin{align*} \int_{X\setminus D} \chi_a \langle \mathcal{D}(f) , \mathcal{D} (g) \rangle \omega^n
\end{align*}
by terms involving at least one derivative of $\chi_a$, hence is an integral over $K_a = \{ x \in X \setminus D : t(x) \in [a,a+1] \}$.

Since $f \in C^{4,\alpha}_{\eta}$, we have that $| \mathcal{D}(f)| \leq c e^{\eta t}$ for some $c >0$. Also, the derivative of $\chi_a$ is bounded on $[a,a+1]$ independently of $a$. Finally, we have that by possibly increasing $c$, $g$ and the norm of its gradient is bounded by $c e^{\eta t}$ as well.
Thus
\begin{align*} | \int_{X \setminus D}  \chi_a \mathcal{D}^* \mathcal{D} (f) g \omega^n - \int_{X\setminus D} \langle \mathcal{D}(f) , \mathcal{D} (\chi_a g)\rangle \omega^n | \leq C \int_{K_a} e^{2 \eta t}  \omega^n
\end{align*}
for some $C > 0$. This latter integral is mutually bounded with 
\begin{align*} \int_a^{a+1} e^{2\eta t - t} dt,
\end{align*}
which goes to zero as $a \rightarrow \infty$ precisely if $\eta < \frac{1}{2}$. Thus the limits in \eqref{lim1} and \eqref{lim2} agree, and the proof is complete.
\end{proof}
We can now prove Proposition \ref{kercokerlem}. First note that $0$ is an indicial root. Indeed, the constant functions are in the kernel of $\mathcal{D}^* \mathcal{D}$ and are in $C^{k,\alpha}_{\delta}$ precisely when $\delta \geq 0$, so the kernel changes at $\delta = 0$. By the duality between the kernel and cokernel for weights $\delta$ and $1- \delta$, it follows that $1$ is also an indicial root. Moreover, Auvray showed in \cite{auvray14b} that there are no indicial roots in $(0,1)$. Thus there exists a $\kappa > 0$ such that the kernel and cokernel of $\mathcal{D}^* \mathcal{D}$ is constant in the intervals stated.

We first establish the claim for $\eta \in (0,1)$. If $f \in  C^{4,\alpha}_{\eta} (X \setminus D)$ with $\eta < \frac{1}{2}$, we may apply Lemma \ref{ptintbyparts} to conclude that
\begin{align*} \int_{X \setminus D} \mathcal{D}^* \mathcal{D} (f) f  = \int_{X \setminus D} | \mathcal{D} f |^2.
\end{align*}
Thus if $f \in \textnormal{Ker } \mathcal{D}^* \mathcal{D}$, we have that $f \in \textnormal{Ker } \mathcal{D} $. This choice of weights means that the holomorphic vector field $\mathfrak{X}_f$ associated to $f$ then is in $L^2$. Thus it follows from Lemma \ref{parallelvfslem} that $\mathfrak{X}_f \in \mathfrak{h}_{//}^D$ and hence $f \in \overline{\mathfrak{h}_{//}^D} $. Since the elements of $\overline{\mathfrak{h}_{//}^D}$ are in $C^{4,\alpha}_{\eta}$ for any $\eta>0$, it follows that the kernel is as stated for $\eta \in (0,\frac{1}{2})$. Since there are no indicial roots in $(0,1)$, the same conclusion then holds for all $\eta \in (0,1)$.

For $\eta \in (- \kappa, 0)$ the kernel is strictly smaller, since the constants are in $\overline{\mathfrak{h}_{//}^D}$, but not in the domain of $\mathcal{D}^* \mathcal{D}$ for these weights. For these weights the associated holomorphic vector field has to have norm in the order of $e^{\eta' t}$ with $\eta' \leq \eta$. Hence if $f \in \textnormal{Ker } \mathcal{D}_{C^{4,\alpha}_{\eta}}$ with $\eta < 0$, we have to have that $\mathfrak{X}_f^D$, the induced vector field on $D$, is trivial, i.e. $f \in \overline{\mathfrak{h}_0^{D}}$. Indeed, in taking the norm of a vector field 
\begin{align*} \sum_i \sigma_i \frac{\partial}{\partial z_i}
\end{align*}
with $\sigma_1 (0) = 0$, we have, in the model case, that the contribution from $\sigma_1 \frac{\partial}{\partial z_1}$ is
\begin{align*} g_{1 \overline{1}} |\sigma_1|^2 &= O( \frac{|z_1|^2}{ |z_1|^2 \log^2 (|z_1|^2)} ) \\
&= O(e^{-2t}).
\end{align*}
For $\sigma_i \frac{\partial}{\partial z_i}$ with $i>1$, the contribution to the norm is $O(1)$. Since the general case is mutually bounded with this it follows that for $f$ to lie in $C^{4,\alpha}_{\eta}$ with $-1 \leq \eta<0$, one necessarily has to have $\sigma_i = 0$ for all $i>1$, and then
\begin{align*} \| \mathfrak{X}_f \| = O(e^{-t}),
\end{align*}
as required. Note that since $\overline{\mathfrak{h}_0^{D}}$ also contains the constants, the codimension of $\textnormal{Ker } \mathcal{D}^* \mathcal{D}$ in $\overline{\mathfrak{h}_0^{D}}$ is at least one.

By \cite[Thm. 1.4]{lockhartmcowen85}, the index in this range of weights equals the index in the local case, which Auvray showed in \cite[Lem. 3.10]{auvray14b} is $- \textnormal{dim Ker } \mathcal{D}_D^* \mathcal{D}_{D}$, i.e. $ -(r+1)$. Note that the dimension of potentials for holomorphic vector fields on $D$ induced by a vector field tangent to $D$ on $X$ is $s+1$. Also, 
\begin{align*} \textnormal{ Im} (\mathcal{D}^* \mathcal{D}_{C^{4,\alpha}_{\eta}}) \subseteq \overline{\mathfrak{h}_{//}^D}^{\perp}
\end{align*}
Thus if we can exhibit at least $r-s$ linearly independent elements in $C^{0,\alpha}_{\eta} \cap \overline{\mathfrak{h}_{//}^D}^{\perp}$ that are not in the image of $\mathcal{D}^* \mathcal{D}$ on $C^{4,\alpha}_{\eta}$ and which are linearly independent of the image as well, then we have found the full cokernel of $\mathcal{D}^* \mathcal{D}$, because then our reduction of the kernel above implies that
\begin{align*} \textnormal{ind } (\mathcal{D}^* \mathcal{D}_{C^{4,\alpha}_{\eta}}) & =  \textnormal{ dim Ker } \mathcal{D}_{C^{4,\alpha}_{\eta}} - \textnormal{dim Coker } \mathcal{D}_{C^{4,\alpha}_{\eta}} \\
& \leq \big(\textnormal{dim } ( \overline{\mathfrak{h}_{//}^D}) - (s+1)\big) - \big( \textnormal{dim } ( \overline{\mathfrak{h}_{//}^D}) + r-s\big) \\
&= -(r+1),
\end{align*}
and so the kernel cannot be smaller, nor can the cokernel be any larger.

In \cite{auvray14b}, Auvray showed that 
\begin{align*} \mathcal{D}^* \mathcal{D} (\psi_f) \in C^{0,\alpha}_{-1},
\end{align*}
for any $f\in \ker \mathcal{D}_D^*\mathcal{D}_D$. If this is in the image of $\mathcal{D}^* \mathcal{D}$ on $C^{4,\alpha}_{\eta}$ with $\eta <0$, say 
\begin{align*} \mathcal{D}^* \mathcal{D} (\psi) =  \mathcal{D}^* \mathcal{D} (v),
\end{align*}
then 
\begin{align*} \psi - v \in  \textnormal{Ker } \mathcal{D}^* \mathcal{D}_{C^{4,\alpha}_{\eta'}}
\end{align*}
for any $\eta' > 0$. By the previous part, this implies $\psi - v = h \in  \overline{\mathfrak{h}_{//}^D}$. 

We now invoke Lemma \ref{psiperturblem} below which says that $\psi - h \in C^{4,\alpha}_{\eta}$ for some weight $\eta <0$ if and only if $f$ is a potential for the vector field on $D$ induced by $h$, under a suitable normalisation. This completes the proof, because then for each $\psi$ coming from an $f$ inducing a holomorphic vector field on $D$ which also is induced by a holomorphic vector tangent to $D$, we can choose a $h \in \overline{\mathfrak{h}_{//}^D}$ such that $\psi - h \in C^{4,\alpha}_{\eta}$ with $\eta < 0$ and $\mathcal{D}^* \mathcal{D} (\psi - h) = \mathcal{D}^* \mathcal{D} (\psi)$. Hence 
\begin{align*} \mathcal{D}^* \mathcal{D} (\psi) \in \textnormal{Im } \mathcal{D}^* \mathcal{D}_{C^{4,\alpha}_{\eta}}
\end{align*}
if and only if $f$ induces a holomorphic vector field on $D$ also induced by a holomorphic vector field on $X$ tangent to $D$. 

\begin{lem}\label{psiperturblem} For $f \in \textnormal{Ker } \mathcal{D}^*_D \mathcal{D}_D$, let $\psi = \psi_f = \chi ( \Pi^* f + d \varphi( \Pi^*\mathfrak{X}_f ) )$. Then there is a $h \in \overline{\mathfrak{h}_{//}^D}$ such that $\psi - h \in C^{4,\alpha}_{\eta}$ for some $\eta < 0$ if and only if the associated vector field $\mathfrak{X}_f$ of $f$ on $D$ is induced by a vector field in $\mathfrak{h}_{//}^D$.
\end{lem}

\begin{proof} Let $Z \in \mathfrak{h}_{//}^D$ be a vector field on $X$ such that $Z_{|D} = \mathfrak{X}_f$. Let $h$ be the potential for $Z$ with respect to the Poincar\'e type metric $\omega = \omega_0 + i \partial \overline{\partial} \varphi$. Let $h_0$ be the corresponding potential with respect to the smooth background metric $\omega_0$. Using \cite[Prop. 1.2]{auvray14}, we then have that
\begin{align*} h &= h_0 +  d\varphi (Z) \\
&= h_0 + ( \nabla^{\omega_0} h_0) \cdot \varphi.
\end{align*}
Up to a constant, $f$ is the restriction of $h_0$ to $D$, and so we can renormalise $h$ to assume this is true. It then follows that 
\begin{align*} \chi \Pi^*f_0 - h_0 \in C^{4,\alpha}_{\eta}
\end{align*}
for some $\eta < 0$. Hence 
\begin{align*} \psi_f - h \in C^{4,\alpha}_{\eta}
\end{align*}
for some $\eta < 0$, too.

Conversely, suppose $\psi_f - h \in C^{4,\alpha}_{\eta}$ for some $\eta <0$. Then the associated vector fields are also equal to order $e^{\eta t}$ and so their restrictions to $D$ must be equal. 
\end{proof}

A final consequence of our explicit analysis that we want to mention now is that the extremal vector field of an extremal Poincar\'e type metric is the restriction to $X \setminus D$ of a vector field on $X$ tangent to $D$.
\begin{lem}\label{extvfistangent} Let $\omega \in \Omega$ be Poincar\'e type metric on $X \setminus D$. Then $$S(\omega) \in C^{k,\alpha} (X \setminus D)$$ for any $k$ and $\alpha$. In particular, if $\omega$ is an extremal metric, then $S(\omega) \in \overline{\mathfrak{h}_{//}^{D}}$.
\end{lem}
\begin{proof} Auvray showed in \cite[Prop. 1.6]{auvray17} that the Ricci form $\rho_{\omega}$ associated to $\omega$ is bounded at any order, i.e. lies in the space $C^{\infty}(\Lambda^{1,1}, X\setminus D)$. Similarly, so does $\omega$, by definition of a Poincar\'e type metric. Hence both $\rho_{\omega} \wedge \omega^{n-1}$ and $\omega^n$ lie in $C^{\infty}(\Lambda^{n,n}, X\setminus D)$, and so the scalar curvature function 
\begin{align*} S(\omega) = \frac{n \rho_{\omega} \wedge \omega^{n-1}}{\omega^n}
\end{align*}
lies in $C^{\infty}(X\setminus D)$, as required. It therefore follows that if $\omega$ is extremal, $S(\omega)$ lies in the kernel of the Lichnerowicz operator and by the above also in $C^{k,\alpha}_{\eta}$ for any $\eta \geq 0$. Taking e.g. $\eta = \frac{1}{2}$, Proposition \ref{kercokerlem} then implies that $S(\omega) \in \overline{\mathfrak{h}_{//}^{D}}$.
\end{proof}

\subsection{The modified H\"older spaces}\label{modifiedholder}

Since we work with H\"older spaces in which not all of the potentials for holomorphic vector fields are contained, it will be convenient to modify these spaces slightly, which we do in this section. In general we could pullback functions $f$ from $D$ to $X \setminus D$ by using the tubular neighbourhood discussed in Section \ref{fnspaces}. We choose a cutoff function $\chi$ only depending on the variable $t$ and consider $\chi \Pi^* (f )$, where $\Pi$ is the (local) fibration map. These functions all lie in $C^{k,\alpha}_0 (X \setminus D)$ if $f \in C^{k,\alpha} (D)$, and we will need to include some of these functions when solving the blow-up problem.

We begin with a Lemma which finds a function whose image via the Lichnerowicz operator is the pulled back function, for functions on $D$ that are potentials for holomorphic vector fields on $D$.
\begin{lem}\label{pullbackperturb}Let $\omega$ be a Poincar\'e type metric on $X \setminus D$ satisfying \eqref{metricasymptotics}. Then there exists $\eta < 0$ such that for all $\tilde{f} \in \textnormal{Ker }  \mathcal{D}_D^* \mathcal{D}_{D}$ there exists $\sigma \in C^{0,\alpha}_{\eta} (X \setminus D), \phi \in C^{4,\alpha} (D)$ and $f \in \textnormal{Ker }  \mathcal{D}_D^* \mathcal{D}_{D}$ such that
\begin{align}\label{perturbtoetaspace} \mathcal{D}^*_{\omega} \mathcal{D}_{\omega} (\chi \Pi^* (\phi) +  t \chi \Pi^*f ) = \chi \Pi^*\tilde{f} + \sigma.
\end{align}
Moreover, $f$ is unique and $\phi$ is unique up to an element of $\textnormal{Ker } \mathcal{D}_D^* \mathcal{D}_D$. Finally, if $\widetilde{f} = 1$, we can take $f=1$ and $\phi =0$.
\end{lem}
\begin{proof}
We begin with the case of the model metric on $\Delta^* \times D$. Recall from \eqref{modellich} that the Lichnerowicz operator $L_{\textnormal{mod}}$ then is given by
\begin{align*}f &\mapsto  \frac{1}{2} \big( \frac{\partial^2 }{\partial t^2} - \frac{\partial }{\partial t} \big)^2 (f) -  \big( \frac{\partial^2 }{\partial t^2} - \frac{\partial }{\partial t} \big) (\Delta_D f)  - \big( \frac{\partial^2 }{\partial t^2} - \frac{\partial }{\partial t} \big) (f)  + \mathcal{D}_D^* \mathcal{D}_D f.
\end{align*}
Then for the function $t\cdot \Pi^*f$, with $f$ a function on $D$, we get that
\begin{align*} L_{\textnormal{mod}} (t\cdot f) = \Pi^* ( f +  \Delta_D f + t \mathcal{D}_D^* \mathcal{D}_D f ).
\end{align*}
In particular, for all $f\in \textnormal{Ker }  \mathcal{D}_D^* \mathcal{D}_D$, we have
\begin{align*} L_{\textnormal{mod}} (t\cdot \Pi^*f) = \Pi^*(  f + \Delta_D (f)).
\end{align*}  
Now, note that since 
\begin{align*} \int_D(  f +  \Delta_D (f) )\cdot f = \int_D |f|^2 + |df|^2,
\end{align*}
we have that $f + \Delta_D (f) \notin \big( \textnormal{Ker } \mathcal{D}_D^* \mathcal{D}_D \big)^{\perp} = \textnormal{Im }\mathcal{D}_D^* \mathcal{D}_D$.

In fact, if $f_1, \cdots, f_r$ is a basis of $\textnormal{Ker }  \mathcal{D}_D^* \mathcal{D}_D$, then $f_i +  \Delta_D f_i$ form a basis of a complementary space to $\textnormal{Im }\mathcal{D}_D^* \mathcal{D}_D$. This follows by integrating by parts again, since if $\mathcal{D}^* \mathcal{D}_D (\phi)=f + \Delta_D f$, then
\begin{align*}  \int_D  |f|^2 + |df|^2 &= \int_D(f + \Delta_D (f) )\cdot f  \\
&= \int_D \mathcal{D}^*_D \mathcal{D}_D (\phi) f \\
&= \int_D \phi \mathcal{D}^*_D \mathcal{D}_D (f)  \\
&= 0,
\end{align*}
as $f \in \textnormal{Ker } \mathcal{D}^*_D \mathcal{D}_D$. In particular, given any $\psi \in C^{0,\alpha}(D)$, there exists $f \in \textnormal{Ker }  \mathcal{D}^*_D \mathcal{D}_D$ and $\phi \in C^{4,\alpha} (D)$ such that
\begin{align*} \mathcal{D}^*_D \mathcal{D}_D (\phi) +  f +  \Delta_D f= \psi.
\end{align*}

Applying this to $\tilde{f}$ in the model case, we have that there are $\phi$ and $f$ such that
\begin{align*} \Pi^*  \tilde{f} &= \Pi^* (\mathcal{D}^*_D \mathcal{D}_D (\phi) + f + \Delta_D f ) \\
&= \mathcal{D}^* \mathcal{D} (\Pi^* \phi + t \cdot \Pi^* f).
\end{align*}
We can then let $\sigma$ be given by
\begin{align*} \sigma = \chi \mathcal{D}^* \mathcal{D} (\Pi^* \phi + t \cdot \Pi^* f) - \mathcal{D}^* \mathcal{D} (\chi \Pi^* \phi + t \cdot \chi \Pi^* f) ,
\end{align*}
which lies in $C^{0,\alpha}_{\eta}$ for \textit{any} $\eta$, since $\sigma$ vanishes identically in a neigbhourhood of $D$.

So far we assumed that in the assumption \eqref{metricasymptotics} on the asymptotics we had $a=1$. For other values of $a$, we have that there is an $\eta < 0$ such that the estimates \eqref{modellich} hold, but where we are using the model operator for $a \omega_{cusp}$ instead of $\omega_{cusp}$. This affects the coefficients of $f_i$ and $\Delta_D f_i$ above, but it does not affect the conclusion, because we still obtain a positive combination of $f_i$ and $\Delta_D f_i$. In particular $\mathcal{D}^* \mathcal{D} (\chi \Pi^* (\phi) +  t \chi \Pi^*f  )$ agrees with $\mathcal{D}^*\mathcal{D}_{\textnormal{mod}}(\chi \Pi^* (\phi) +  t \chi \Pi^*f )$ up to an element of order $e^{ \eta t}$. Hence we can always solve our equation up to an error of order $e^{\eta t}$ with $\eta <0$, as required. The statement about $\widetilde{f} = 1$ follows because in that case, if we take $f=1$ too, then $\Delta_D (f) =0$.
\end{proof}

 We now decompose $\overline{\mathfrak{h}_D}$ as 
\begin{align}\label{Ddecomp} \overline{\mathfrak{h}_D} = \langle 1 \rangle \oplus V_1 \oplus V_2,
\end{align} 
where $$V_1 = \langle f_1, \cdots, f_s \rangle $$ and $$V_2 = \langle f_{s+1}, \cdots, f_r \rangle. $$ In other words, we have decomposed the potentials for holomorphic vector fields into three pieces: the constants, the potentials for vector fields induced from $X$, and those not induced from $X$, respectively.  Define the linear map $$\varrho : V_1 \oplus V_2 \rightarrow C^{0,\alpha}_{0}$$ by $$\varrho (\widetilde{f} ) = \mathcal{D}^*_{\omega} \mathcal{D}_{\omega} \big( \chi \Pi^* (\phi ) + t \chi \Pi^* (f) \big),$$ where $\phi, f$ are given by Lemma \ref{pullbackperturb} if $\widetilde{f} \in  V_1$, and $$\varrho (\widetilde{f} ) = \mathcal{D}^*_{\omega} \mathcal{D}_{\omega} ( \psi_{\widetilde{f}} ),$$ where $\psi_{\widetilde{f}}$ is given by equation \eqref{psif}, if $\widetilde{f} \in V_2$. 

Note that $$ V_1 \oplus V_2 \cong \mathfrak{h}_D ,$$ by sending a function to the corresponding holomorphic vector field. To slightly ease notation we will in the sequel write that $\varrho$ is a map from $\mathfrak{h}_D$ and we are then using this isomorphism. Under this identification, we then have that the decomposition \eqref{Ddecomp} of $\overline{\mathfrak{h}_D}$ is $\mathbb{R} \times \mathfrak{h}_D$. Also, let $\widetilde{C}^{k,\alpha}_{\eta} = C^{k,\alpha}_{\eta} \oplus \overline{\mathfrak{h}_{//}^D} \subseteq C^{k,\alpha}_0$ for $\eta <0$. It is more convenient for us to map into this space, as it contains all the potentials for holomorphic vector fields. 

As a direct consequence of Proposition \ref{kercokerlem}, we then have
\begin{lem}\label{modifiedkercokerlem} Let $\eta \in (- \kappa ,0)$. Consider the map
\begin{align}\label{phidefn} \Phi  : C^{4,\alpha}_{\eta} (X \setminus D) \times \mathbb{R} \times \mathfrak{h}_D \rightarrow \widetilde{C}^{0,\alpha}_{\eta} (X \setminus D)
\end{align}
given by
\begin{align*}
\Phi (\phi, \lambda, f) = \mathcal{D}_{\omega}^* \mathcal{D}_{\omega} \big( \phi + \lambda t \big) + \varrho (f).
\end{align*}
Then
\begin{align*} \textnormal{Ker } \Phi = \textnormal{Ker } \big( \mathcal{D}^*_{\omega} \mathcal{D}_{\omega} \big)_{C^{4,\alpha}_{\eta}} \times \{0 \} \times \{0 \} 
\end{align*}
and 
\begin{align*}  \textnormal{Im }\Phi = \big( \overline{\mathfrak{h}_{//}^D} \big)^{\perp}.
\end{align*} 
\end{lem}

\section{Linear theory in doubly weighted spaces}\label{lintheory}

In this section we analyse the Fredholm theory of the Lichnerowicz operator in \textit{doubly weighted spaces}, H\"older spaces with weighted norm both near the divisor as discussed earlier, as well as near the blown-up point, with the weight function being the distance to the blown-up point. 

When solving the extremal equation on the blow-up, in order to get uniform estimates, we want to use the radius function around the points as a weight function. We now define a doubly weighted norm on the complement of the points in $X \setminus{D}$ that are to be blown up. Pick $T$-invariant normal coordinates $\underline{z}^j$ at $p_j$, which after scaling can be assumed to be defined for when the norm is at most $2$. We define the doubly weighted H\"older norm $C^{k,\alpha}_{\delta, \eta}$ on $Y = X \setminus \big(D \cup \{p_1, \cdots, p_k \} \big)$ to be
\begin{align*} \| \phi \|_{C^{k,\alpha}_{\delta,\eta} (Y)} = \| \phi \|_{C^{k,\alpha}_{\eta} (V), \omega) } + \sum_i \textnormal{sup}_{r\in(0,\frac{1}{2})} r^{- \delta} \| \phi \|_{C^{k,\alpha}(B_{2r} \setminus B_{r} (p_j), r^{-2} \omega)} ,
\end{align*}
where $V $ is the complement to $\cup_j B_{\frac{1}{2}} (p_j)$ in $X \setminus D$. Here in e.g. $C^{k,\alpha}(B_{2r} \setminus B_{r}, r^{-2} \omega)$, the second entry denotes the metric we are using to compute norms with. Also we let $r_{\varepsilon} = \varepsilon^{\frac{2n-1}{2n+1}}$. 

We then have a similar Fredholm result to Theorem \ref{ptfredholmthm} for the doubly weighted spaces. 
\begin{thm}\label{dwtdfredholm} Let $\omega$ be a Poincar\'e type metric on $X \setminus D$ satisfying equation \ref{metricasymptotics}. Suppose $(\delta,\eta)$ are weights such that $\delta$ is not an indicial root of $\Delta^2$ on $\mathbb{C}^n \setminus \{ 0 \}$, i.e. $\delta \notin \mathbb{Z} \setminus (4-2n,0)$, and $\eta$ is not an indicial root of $\mathcal{D}^*\mathcal{D}$ on the Poincar\'e type weighted spaces. Then 
\begin{align*} \mathcal{D}_{\omega}^* \mathcal{D}_{\omega} : C^{k+4,\alpha}_{\delta, \eta} (Y) \rightarrow C^{k,\alpha}_{\delta-4,\eta} (Y)
\end{align*}
 is Fredholm. Moreover, 
\begin{align}\label{doubleimageintermsofker} \textnormal{Im } \mathcal{D}^* \mathcal{D}_{C^{k+4, \alpha}_{\delta, \eta}} = ( \textnormal{Ker} (\mathcal{D}^*\mathcal{D}_{C^{k+4,\alpha}_{4-2n-\delta,1 - \eta}} ) )^{\perp},
\end{align}
where $\perp$ denotes the orthogonal complement with respect to the $L^2$-inner product and subscripts denote the domains of the operators.
\end{thm}

This follows from the Fredholm theory of Section \ref{fredholmpropssect} together with that of weighted spaces of punctured compact manifolds, see e.g. \cite{lockhartmcowen85}. Indeed, one could use cut-off functions to view a function on $Y$ has having a component on $X \setminus D$ that lies in $C^{k+4, \alpha}_\eta (X \setminus D)$ and a component on $X \setminus \{ p_1, \cdots, p_k \}$ in $C^{4,\alpha}_{\delta} (X \setminus \{ p_1, \cdots, p_k \})$, or the relevant weighted Sobolev spaces. From this one can establish estimates similar to those of Propositions \ref{globalw2kestimate} and \ref{ptholderreg}. Then the result follows by going through the argument of Section \ref{ptfredholmthmpf} again. Note that on the weighted spaces for $X \setminus \{ p_1, \cdots, p_k \}$, the Lichnerowicz operator has a similar characterisation of its image in terms of the orthogonal complement to a complementary weight. There the image of $C^{k+4, \alpha}_{\delta}$ is the orthogonal complement to the kernel of the Lichnerowicz operator acting on $C^{k+4, \alpha}_{4-2n-\delta}$.

We end this section with characterising the (co)-kernel for the weights relevant to us. We can also define modified doubly weighted H\"older spaces, like in Section \ref{modifiedholder}. If we pick the cut-off function $\chi$ to be $0$ sufficiently close the blow-up points, then these functions will not interact with the $\delta$ weights. We can then define the map $$\Phi : C^{4,\alpha}_{\delta, \eta} \times \overline{\mathfrak{h}_D} \rightarrow \widetilde{C}^{0,\alpha}_{\delta - 4, \eta}$$ as in equation \eqref{phidefn}. 
\begin{prop}\label{dwtdexplicit} Suppose that $\delta \in (4-2n, 0)$ and that $\eta \in (- \kappa, 0)$.Then 
\begin{align*}
\textnormal{Ker } \mathcal{D}^* \mathcal{D}_{C^{4,\alpha}_{\eta}} &\subseteq \overline{\mathfrak{h}_0^{D}} \textnormal{ and is of codimension }1 , \\
 C^{0,\alpha}_{\eta} \cap   \overline{\mathfrak{h}_{//}^D}^{\perp} &= \textnormal{ Im} (\mathcal{D}^* \mathcal{D}_{C^{4,\alpha}_{\eta}}) \oplus \langle \mathcal{D}^* \mathcal{D} ( \psi_i ): i \in \{1, \cdots, s \} \rangle  \\
\textnormal{Ker } \Phi &= \textnormal{Ker } \big( \mathcal{D}^*_{\omega} \mathcal{D}_{\omega} \big)_{C^{4,\alpha}_{\eta}} \times \{0 \}  \\
\textnormal{Im }\Phi &= \big( \overline{\mathfrak{h}_{//}^D} \big)^{\perp}.
\end{align*} 
\end{prop}
This follows from Theorem \ref{dwtdfredholm} and that when the weights are in $(4-2n,0)$, the elements of the kernel on the doubly weighted spaces actually extend across the punctured points. Thus such functions can be considered as elements in $C^{k+4, \alpha}_{\eta} ( X \setminus D )$, and the result is then a direct consequence of Proposition \ref{kercokerlem} and Lemma \ref{modifiedkercokerlem}.


\section{Solving the non-linear equation}\label{nonlinearsect}

Having the linear theory in place, we are now ready to solve the extremal equation. We begin by stating the system of equations we would like to solve in order to solve the extremal equation on the blow-up. This is identical to the case of \cite{arezzopacardsinger11}. Let $\BlX$ denote the blow-up of $X$ in the points $p_1, \cdots, p_k$. We identify $D \subseteq X$ with its pull-back to $\BlX$ via the blow-down map, and still denote this $D$.

Let $\ball$ denote the (open) ball of radius $r_{\varepsilon} = \varepsilon^{\frac{2n-1}{2n+1}}$ about $p_j$ in $X \setminus D$, where the radius is measured as the Euclidean distance in some fixed $T$-invariant holomorphic normal coordinates about $p_j$. Let $Y_{\varepsilon}$ denote the complement in $X \setminus D$ of the union of these balls. Let $\Blball$ denote the subset of $\Bl_0 \mathbb{C}^n$ given as the pre-image via the blow-down map of a ball of radius $\frac{R_{\varepsilon}}{\sqrt{a_j}}$ about the origin in $\mathbb{C}^n$, where $R_{\varepsilon} = \varepsilon^{-\frac{2}{2n+1}}$. For each $\varepsilon > 0 $ we have an identification 
\begin{align}\label{connectedsum} \BlX \setminus D \cong Y_{\varepsilon} \coprod_{j=1}^k \Blball / \sim
\end{align} 
of the blow-up with a connected sum, where the equivalence relation $\sim$ is the gluing of the boundary $\partial B^j_{\varepsilon}$ of $B^j_{\varepsilon}$ with the boundary $\partial \Blball$ of $\Blball$ via the coordinate change $$ \underline{z}^j = \varepsilon \sqrt{a_j} \cdot \underline{w}^j.$$ Here $\underline{z}^j =(z_1^j, \cdots, z_n^j)$ are the holomorphic normal coordinates about $p_j$ fixed earlier, and $\underline{w}^j$ are the coordinates on the complement of the exceptional divisor in $\Bl_0 \mathbb{C}^n$ coming from its identification with $\mathbb{C}^n \setminus \{ 0 \}$.

The aim is then to construct extremal metrics on each of the pieces in \eqref{connectedsum} and show that we can match them to sufficiently high order over their common boundary. This is done in several steps. We follow \cite{arezzopacardsinger11} and make an initial perturbation of an approximate metric $\omega_{\varepsilon}$ constructed earlier to obtain a metric which is extremal to a high order (in terms of the distance function to the blown up points). This initial perturbation only depends on the constants $a_j$. Given boundary data, we then perform a second perturbation to construct a metric which extremal to higher order. We then use this metric to construct metrics that are extremal up to a finite dimensional set of obstructions on the two pieces $Y_{\varepsilon}$ and $\cup_{j=1}^{k} \Blball$, parametrised by certain boundary data. We then show that for all sufficiently small $\varepsilon > 0$ we can use these metrics to solve the same boundary value problem and thus solve the extremal equation. This is the content of the next sections.

\begin{rem} In the weighted analysis near the blow-up points, we need to take special care with the case of surfaces as one then needs to work with different weights. This features in both \cite{arezzopacardsinger11} and \cite{szekelyhidi12}. However, the way to approach this is no different in our case than in the compact case. Since our focus is on the new behaviour the Poincar\'e type asymptotics introduce, we will not go further in discussing how to alter the argument for the surface case, and instead simply refer to \cite{arezzopacardsinger11}.
\end{rem}

\subsection{The initial perturbation}
In order to solve the extremal equation, one can make an initial approximate solution in the appropriate class as follows. This step features in both the approach of \cite{arezzopacardsinger11} and \cite{szekelyhidi12}, but we will follow an argument closer to that of the latter. We focus on the case of one point, with the construction for several points simply being that one does exactly the same construction around each point separately, with appropriate scaling.  Around a point $p$ to be blown up, recall that we use $T$-invariant holomorphic normal coordinates to write the K\"ahler form as
\begin{align*} \omega = i \partial \overline{\partial} \bigg( \frac{|z|^2}{2} + \phi (z) \bigg)
\end{align*}
for some $\phi$ which is $O(|z|^4)$. After scaling $\omega$, we can assume the normal coordinates are defined for $|z| \leq 2$ (when blowing up several points we scale $\omega$ so that this holds for all the points and such that all these balls are disjoint).

The Burns-Simanca metric $\zeta$ is a metric on the blow-up $Bl_0 \mathbb{C}^n$ of $\mathbb{C}^n$ at the origin which is scalar flat and asymptotically Euclidean. We write
\begin{align*} \zeta = i \partial \overline{\partial} \bigg( \frac{|w|^2}{2} + \psi (w) \bigg),
\end{align*}
where $w$ is the coordinates away from the exceptional divisor in $\Bl_0 \mathbb{C}^n$ induced from its identification with $\mathbb{C}^n \setminus \{ 0 \}$.

Let $r_{\varepsilon} = \varepsilon^{\frac{2 n -1}{2n+1}}$. By taking a slightly different viewpoint in the connected sum construction, we consider $Bl_p X$ as the manifold obtained by gluing the complement of a ball around $p$ in $X$ with a neighbourhood of the expectional divisor in the blow-up of $\mathbb{C}^n$ in the origin. This is achieved by identifying the annulus $B_{2r_{\varepsilon}} \setminus B_{r_{\varepsilon}} $ with a corresponding annulus around the exceptional divisor on $Bl_0 \mathbb{C}^n$, using the coordinate transformation $w = \varepsilon^{-1} z$. The approximate solution will be constructed by gluing $\omega$ and $\zeta$ on this annulus. 

Let $\gamma$ be a cut-off function $\mathbb{R} \rightarrow [0,1]$ with 
\begin{align*} \gamma (x) = 0, & \textnormal{ }x < 1, \\
\gamma (x) = 1, &\textnormal{ } x > 2.
\end{align*}
Define $\gamma_1$ to be
\begin{align*} \gamma_1 (r) = \gamma (\frac{r}{r_{\varepsilon}}),
\end{align*}
and let $\gamma_2 = 1- \gamma_1$. We define the approximate solution to be $\omega$ on the complement of $B_1$ and 
\begin{align*} i \partial \overline{\partial} \big(\frac{|z|^2}{2} + \gamma_1 (|z|) \phi (z) +  \varepsilon^2 \gamma_2 (|z|) \psi (\varepsilon^{-1} z)  \big)
\end{align*}
on $B_1 \setminus B_{\varepsilon}$. 

Since $\frac{2n-1}{2n+1} < 1$, we have that $r_{\varepsilon} > \varepsilon$ and so $B_{\varepsilon} \subseteq B_{r_{\varepsilon}}$. On $B_{r_{\varepsilon}}$, we have $\gamma_1=0$ and $\gamma_2 =1$, so that the approximate solution is 
\begin{align}\label{scaledclass} i \varepsilon^2 \partial \overline{\partial} ( \frac{|\varepsilon^{-1} z|^2}{2} + \psi (\varepsilon^{-1} z) ) &= \varepsilon^2 \zeta.
\end{align}
So in the pre-image of $B_{\varepsilon}$ in $Bl_p M$ under the blow-down map, we let the approximate solution equal the scaled Burns-Simanca metric $\varepsilon^2 \zeta$. 

When blowing up several points, we do not want to impose that the volume of all the exceptional divisors are equal. The change of coordinates is now $z^j = \varepsilon \sqrt{a_j} \cdot w^j$,  so we instead use the scaling 
\begin{align*} i \partial \overline{\partial} \big(\frac{|z|^2}{2} + \gamma_1 (|z|) \phi (z) +  a_j \varepsilon^2 \gamma_2 (|z|) \psi (\varepsilon^{-1} z)  \big)
\end{align*}
in the annular region around the point $p_j$, so that the approximate solution is in the correct class, i.e. we obtain $a_j \varepsilon^2 \zeta$ in \eqref{scaledclass} near $p_j$.

We now wish to find a better approximate solution to the extremal equation on $\BlX \setminus D$. We stress that in contrast to the next steps, finding the function $\Gamma = \Gamma_{a_1, \cdots, a_k}$ achieving this only depends on the direction into the K\"ahler cone we are going, i.e. only the $a_j$, and does not involve any boundary data. 

To find a better approximate solution, we need to match the metric glued in from $X 
\setminus D$ with the scaled Burns-Simanca metrics on around each point to higher order. We cannot find such a metric on the whole of $X \setminus D$, but under our assumptions we can achieve it on the complement on the blown-up points, applying the linear theory of Section \ref{lintheory}. 

When using the coordinate identifications above, the Burns-Simanca metric $a_j \zeta$ has an asymptotic expansion $$ a_j \zeta = \ddb \bigg( \frac{| z |^2}{2} - a_j^{n-1} |z|^{4-2n} + \phi (z) \bigg) $$ where $\phi$ is $O(| z |^{6-2n})$ when $\textnormal{dim}(X) > 3$ and $O(\log (| z |))$ when $\textnormal{dim}(X) = 3$, and $$a_j \zeta = \ddb \bigg( \frac{| z |^2}{2} - a_j \log (z) + \phi(| z |) \bigg)$$ with $\phi = O(1)$ when $\textnormal{dim}(X) =2$. Here $z = \underline{z}^j$ is the holomorphic normal coordinates about $p_j$. Thus to match $\omega + \ddb \big( \varepsilon^{2n-2} \Gamma \big)$ up with $a_j \varepsilon^2  \zeta$ to higher order, we wish to find a solution to 
\begin{align} \mathcal{D}^* \mathcal{D} \big( \Gamma \big) = h - \sum_{j=1}^k a_j^{n-1} \delta_{p_j} ,
\end{align}
where $h$ is a holomorphy potential. In the compact case, what we require is that $h$ is a potential for a vector field in $\mathfrak{t}$. In our case, $\overline{\mathfrak{h}} \not\subseteq C^{0, \alpha}_{\eta}$, so we cannot simply do this. However, using the modified H\"older spaces of Section \ref{modifiedholder}, we can find a function $\Gamma$ decaying near the divisor, an average zero holomorphy potential $f_{\Gamma} \in \mathfrak{h}_D$, a constant $\lambda_{\Gamma}$, and  $h_{\Gamma} \in  \mathfrak{h}_{//}^D$ such that this holds, i.e. we can solve 
\begin{align}\label{gammadefn} \mathcal{D}^* \mathcal{D} \big( \Gamma + \lambda_{\Gamma} t \big) = h_{\Gamma}  + \varrho (f_{\Gamma}) - \sum_{j=1}^k a_j^{n-1} \delta_{p_j} .
\end{align}
This follows from Proposition \ref{dwtdexplicit}. 

Since we also want this to give an approximate solution to the extremal equation on the blow-up, we need to be able to lift $h$ to a holomorphic vector field on the blow-up. This is only possible when $h$ induces a holomorphic vector field that lies in the subalgebra $\mathfrak{t}$ of $\mathfrak{h}_{//}^D$, i.e. if the requirement \eqref{condbala} in Theorem \ref{mainblthm} hold. 

Note that the functions induced from the vector fields on $D$ do lift to the blow-up, but are not potentials for holomorphic vector fields on the blow-up. We will see that these functions give the new obstructions to obtaining an extremal metric on the blow-up in the Poincar\'e type case.

\begin{rem} In contrast to imposing that $h \in \overline{\mathfrak{t}}$, we do not impose that $f=0$ now, because in the Arezzo-Pacard type argument we will need to let the divisor volumes vary. Therefore it is only at the end, when we know the actual divisor volumes, that we will check that no term like $\varrho(f)$ was needed. This will use the assumption \eqref{extremalvf}. A posteriori we see that $f$ would have to be $0$ for the classes we consider, by differentiating the family $\mathfrak{X}_{\varepsilon}$ with respect to $\varepsilon$.
\end{rem}

\subsection{The second perturbation}

The next step in the proof is to construct an even better approximate solution near the gluing region to the extremal equation, given boundary data on the common boundary of the pieces in the connected sum presentation \eqref{connectedsum} of $\BlX \setminus D$. 

The highest order term of the Lichnerowicz operator agrees with that of the bi-Laplacian $\Delta^2$. In the gluing region, the metrics are approximately Euclidean, and so in this region the metric Laplacian agrees with the usual Laplacian to high order. Using the $\varepsilon$-dependent identification of the fixed annular region $B_2 \setminus B_{\frac{1}{2}}$ with such a region either in the punctured manifold $Y$ or $\Bl_0 \mathbb{C}^n$, we will get the approximate solutions to match up to higher order by pulling back functions that are biharmonic with respect to the Euclidean Laplacian on $B_2 \setminus B_{\frac{1}{2}}$.

We will pull back functions given by the following 
\begin{prop}[{\cite[Prop. 5.3.1]{arezzopacardsinger11}}]\label{secondapprox} Suppose $\upsilon \in C^{4,\alpha} \big( \partial B_1 \big) $ and $\varsigma \in C^{2, \alpha}\big( \partial B_1 \big)$. There is a constant $C>0$ such that: 

If $$ \int_{\partial B_1} 4n \upsilon - \varsigma = 0$$ 
then there exists a biharmonic function $V \in C^{4, \alpha}_1 (B_1 \setminus \{ 0 \} )$ such that 
\begin{align*} V =& \upsilon \\
\Delta V=& \varsigma
\end{align*}
on $\partial B_1$, and
\begin{align*} \| V \|_{C^{4,\alpha}_1 ( B_1 \setminus \{ 0\} )} \leq C \big(  \| \upsilon \|_{C^{4,\alpha} ( \partial B_1 )} +  \| \varsigma \|_{C^{2,\alpha}( \partial B_1  )}  \big) .
\end{align*}

If $$ \int_{\partial B_1} k = 0$$ 
then there exists a biharmonic function $W \in C^{4, \alpha}_{3-2n} (\mathbb{C}^n \setminus B_1 )$ such that 
\begin{align*} W =& h \\
\Delta W=& k
\end{align*}
on $\partial B_1$, and
\begin{align*} \| W \|_{C^{4,\alpha}_1 (\mathbb{C}^n \setminus B_1 )} \leq C \big(  \| h \|_{C^{4,\alpha} ( \partial B_1 )} +  \| k \|_{C^{2,\alpha}( \partial B_1  )}  \big) .
\end{align*}

Moreover, if $\upsilon, \varsigma$ are torus-invariant with respect to the action of some torus contained in $U(n)$, then so are $V$ and $W$.
\end{prop}
Since this is a result for the Euclidean Laplacian, we do not require any modification to the result in Arezzo-Pacard-Singer. 

We end this section by explaining how we will pull $V$ and $W$ back to the preimage of balls $\Bl_0 \mathbb{C}^n$ and $Y_{\varepsilon}$, respectively, in order to create better approximate solutions. For the former, assume that $\upsilon$ and $\varsigma$ are $T$-invariant functions on $\partial B_1$ satisfying $\int 4n \upsilon - \varsigma = 0$. Given a positive parameter $a >0$, we define a $T$-invariant function $V_{\varepsilon, a}$ on $\Bl_0 \overline{B}_{\frac{R_{\varepsilon}}{a}}$ as follows, where we recall $R_{\varepsilon} = \varepsilon^{-\frac{2}{2n+1}}$. 

Let $\chi$ be a $T$-invariant cut-off function vanishing on $\Bl_0 B_{1} $ and equal to $1$ in the complement of $\Bl_0 B_2$ in $\Bl_0 \mathbb{C}^n$. We will let $V_{\varepsilon,a}$ be the function on $\Bl_0 \overline{B}_{\frac{R_{\varepsilon}}{a}}$ that vanishes inside $\Bl_0 B_{1}$ and which outside this region satisfies
\begin{align}\label{Vepsdefn} V_{\varepsilon, a} (w) = \chi (w) V (a \frac{w}{R_{\varepsilon}}),
\end{align}
where $V$ is the function given by Proposition \ref{secondapprox}.

Next, we suppose we are given $T$-invariant functions $\upsilon_1, \cdots, \upsilon_k$ and $\varsigma_1, \cdots , \varsigma_k$ on $\partial B_1$ and that $\int \varsigma_j = 0$ for all $j$. Then we define a $T$-invariant function $W_{\varepsilon}$ on $Y_{\varepsilon}$ as follows. Let $\chi_j$ be a $T$-invariant cut-off function equal to $0$ outside $B_{2}^j (p_j)$ and equal to $1$ in $B^j_1 (p_j)$. We will let $W_{\varepsilon}$ vanish outside $B_{2}^j$ and on $B^j_2 \setminus B_{r_{\varepsilon}}^j$ we let 
\begin{align}\label{Wepsdefn} W_{\varepsilon} (z) = \chi_j (z) W (\frac{z}{r_{\varepsilon}}),
\end{align}
where $W$ is the function given by Proposition \ref{secondapprox}.

\subsection{Constructing extremal metrics on the two pieces}

As mentioned before, we wish to solve a boundary value problem on $Y_{\varepsilon}$ and on each of $ \Blball$ for the extremal equation. In reality we will solve a more general equation, because of the additional cokernel elements in the Poincar\'e type weighted space. Since the points blown up are not on the divisor $D$, the construction of such metrics on $\Blball$ is identical to the construction in \cite{arezzopacardsinger11}. We begin this section by stating our assumptions, before recalling these results of Arezzo-Pacard-Singer, and then prove the analogous result for the Poincar\'e type piece $Y_{\varepsilon}$.

As before, for each $\varepsilon$, we will pull functions back to a fixed annular region $A = B_2 \setminus B_{\frac{1}{2}} \subseteq \mathbb{C}^n$, where $B_r$ is the ball of radius $r$ in $\mathbb{C}^n$. When mapping to this annular region, the points outside of $B_1$ correspond to points in $Y_{\varepsilon}$ and the points in $\overline{B}_1$ lie in one of the $\Blball$. We will fix data on the boundary where the two regions meet that are sufficiently small in the weighted norm. 

We begin with the case of the construction of extremal metrics on the blow-up of all sufficently large balls in $\mathbb{C}^n$. This does not involve the Poincar\'e type behaviour and so is exactly as for Arezzo-Pacard-Singer. Recall that $R_{\varepsilon} = \varepsilon^{- \frac{2}{2n+1}}$. Suppose that $\upsilon \in C^{4,\alpha} (\partial B_1)$ and $\varsigma \in C^{2,\alpha} (\partial B_1)$ are torus-invariant functions satisfying 
\begin{align}\label{boundarycondin} \| \upsilon \|_{C^{4,\alpha}} + \| \varsigma \|_{C^{2,\alpha}} \leq \tau R_{\varepsilon}^4
\end{align}
for some $\tau>0$ that is to be determined. Suppose further that $$ \int_{\partial B_1} 4n\upsilon - \varsigma = 0 .$$

Provided this condition is satisfied, there are extremal metrics with prescribed extremal vector field $\mathfrak{X} \in \mathfrak{t}$, and prescribed average on the boundary, on the blow-up of all sufficently large balls in $\mathbb{C}^n$. 

\begin{prop}[{\cite[Prop.6.2.1]{arezzopacardsinger11}}]\label{interiorsoln} Let $\mathfrak{X} \in \mathfrak{t}$ and $\nu \in \mathbb{R}$. There is a $c>0$ and for every $\tau > 0$ there is a $\varepsilon_{\tau} > 0$ such that if $\varepsilon \in (0, \varepsilon_{\tau})$ then for any $\upsilon,\varsigma$ satisfying \eqref{boundarycondin}, there is a $T$-invariant function $\phi_{\varepsilon, a} \in C^{4,\alpha} \big( \Bl_0 \overline{B}_{\frac{R_{\varepsilon}}{a}} \big)$ such that 
\begin{align*} \zeta_{\varepsilon, a} = a^2 \zeta + \ddb \big( V_{\varepsilon, a} + \phi_{\varepsilon, a} \big)
\end{align*}
is K\"ahler, extremal with extremal vector field $\varepsilon^4 \mathfrak{X}$, and such that the scalar curvature $S_{\varepsilon , a} = S \big(\zeta_{\varepsilon, a} \big)$ satisfies $$\int_{\partial B_1} S_{\varepsilon , a} \big( R_{\varepsilon} \frac{x}{a}  \big) dx = \nu \varepsilon^2 |\partial B_1| .$$ Further, $$ \| \phi_{\varepsilon, a} \big( R_{\varepsilon} \frac{x}{a} \big) \|_{C^{4,\alpha} (\overline{B}_1 \setminus B_{\frac{1}{2}})} \leq c R_{\varepsilon}^{3-2n}.$$

If $\phi_{\varepsilon, a}$ and $\widetilde{\phi}_{\varepsilon, a'}$ are determined by the data $\upsilon, \varsigma, \mathfrak{X}, \nu$ and $\widetilde{\upsilon}, \widetilde{\varsigma}, \widetilde{\mathfrak{X}}, \widetilde{\nu}$, respectively, then for $\delta \in (0,1)$,
\begin{align}\label{intvariationbnd} & \| \phi_{\varepsilon, a} \big( R_{\varepsilon} \frac{x}{a} \big) - \widetilde{\phi}_{\varepsilon, a'} \big( R_{\varepsilon} \frac{x}{a'} \big) \|_{C^{4,\alpha} (\overline{B}_1 \setminus B_{\frac{1}{2}})} \nonumber \\
\leq c_{\tau} \bigg( & R_{\varepsilon}^{1-\delta} \| \upsilon - \widetilde{\upsilon} \|_{C^{4,\alpha}} + R_{\varepsilon}^{1-\delta} \| \varsigma - \widetilde{\varsigma} \|_{C^{2,\alpha}}  \\
& + R_{\varepsilon}^{3-2n} |\nu - \widetilde{\nu}| + R_{\varepsilon}^{3-2n} | a - \widetilde{a}| + R_{\varepsilon}^{4-4n} \| \mathfrak{X} - \widetilde{\mathfrak{X}} \| \bigg) \nonumber,
\end{align}
where $c_{\tau}$ depends on $\tau$, a uniform bound on the norms of $\nu, \nu'$ and the norms of the vector fields $\mathfrak{X}, \widetilde{\mathfrak{X}}$, as well as a uniform bound $$ a_0 \leq a \leq a_1$$ for $a$ and $a'$, where $a_0, a_1 >0.$
\end{prop}

We now turn to the case of constructing extremal metrics away from the blown-up points. Suppose $\upsilon_j \in C^{4,\alpha} \big( \partial B_1 \big)$ and $\varsigma_j \in C^{2,\alpha} \big( \partial B_1 \big)$ are $T$-invariant functions on $\partial B_1$ satisfying 
\begin{align}\label{boundarycondout} \| \upsilon_j \|_{C^{4,\alpha}} + \| \varsigma_j \|_{C^{2,\alpha}} \leq \tau r_{\varepsilon}^4
\end{align}
where, as before, $\tau>0$ is a constant that we will determine at the end of the proof. Moreover, we will assume
\begin{align}\label{kcond} \int_{\partial B_1} \varsigma_j = 0.
\end{align}

The key result of this section is the analogous result to Proposition \ref{interiorsoln} away from the blown up points. We will fix $a_1, \cdots , a_k >0$ and let $\Gamma, \lambda_{\Gamma}, h_{\Gamma}, f_{\Gamma}$  be chosen as in equation \eqref{gammadefn} with respect to this choice of these parameters. 
\begin{prop}\label{exteriorsoln} There is a $c, \theta >0$ and for every $\tau > 0$ there is a $\varepsilon_{\tau} > 0$ such that if $\varepsilon \in (0, \varepsilon_{\tau})$ then for any $\upsilon_j,\varsigma_j$ satisfying \eqref{boundarycondout} and \eqref{kcond}, and for any choice of constants $a_1, \cdots, a_k >0$ there is a $T$-invariant function $\phi_{\varepsilon} \in C_{\eta}^{4,\alpha} \big( Y_{\varepsilon} \big) $, a constant $\lambda_{\varepsilon} \in \mathbb{R}$ , 
an $h_{\varepsilon} \in \overline{\mathfrak{t}}$ and an $f_{\varepsilon} \in \mathfrak{h}_D$ such that
\begin{align*} \omega_{\varepsilon} = \omega + \ddb \big( \varepsilon^{2n-2} \big( \Gamma + \lambda_{\Gamma} t \big) + W_{\varepsilon} + \phi_{\varepsilon} + \lambda_{\varepsilon} t \big)
\end{align*}
is K\"ahler, whose associated vector field has  
potential 
$$ \mathcal{H}_{\varepsilon} + \frac{1}{2} \langle \nabla \big( \mathcal{H}_{\varepsilon} \big) ,  \nabla \big( \varepsilon^{2n-2} \big( \Gamma + \lambda_{\Gamma} t \big) + W_{\varepsilon} + \phi_{\varepsilon} + \lambda_{\varepsilon} t \big) \rangle , $$
where $$\mathcal{H}_{\varepsilon} = S(\omega) + \varepsilon^{2n-2} \big( h_{\Gamma} + \varrho (f_{\Gamma}) \big) + h_{\varepsilon} + \varrho (f_{\varepsilon}).$$ 
The scalar curvature $S_{\varepsilon} = S \big( \omega_{\varepsilon} \big)$ satisfies $$\| h_{\varepsilon} \| + \| f_{\varepsilon} \| + | S_{\varepsilon} - S(\omega)| \leq c \varepsilon^{\theta}.$$ 
Further, $$ \| \phi_{\varepsilon} \big( r_{\varepsilon} x \big) \|_{C^{4,\alpha} (\overline{B}_2 \setminus B_{1})} \leq c r_{\varepsilon}^{4}$$ for all $j$.

If $\phi_{\varepsilon}, \lambda_{\varepsilon}$ and $\widetilde{\phi}_{\varepsilon}, \widetilde{\lambda}_{\varepsilon}$ are determined by the data $\upsilon_j, \varsigma_j$ and $\widetilde{\upsilon}_j, \widetilde{\varsigma}_j$, respectively (but with the same choice of $a_j$), with corresponding holomorphy potentials $h_{\varepsilon}, f_{\varepsilon}$ and $\widetilde{h}_{\varepsilon}, \widetilde{f}_{\varepsilon}$, and scalar curvatures $S_{\varepsilon}, \widetilde{S}_{\varepsilon}$, then
\begin{align}\label{outvariationbnd}  \| h_{\varepsilon} - \widetilde{h}_{\varepsilon} \| +&\| f_{\varepsilon} - \widetilde{f}_{\varepsilon} \| + \| \lambda_{\varepsilon} - \widetilde{\lambda}_{\varepsilon} \| + | S_{\varepsilon} - \widetilde{S}_{\varepsilon} | \\ 
+ \sup_{j} \| \phi_{\varepsilon | \overline{B}^j_{2 r_{\varepsilon}} \setminus B^j_{ r_{\varepsilon}}} & \big( r_{\varepsilon} x \big) - \widetilde{\phi}_{\varepsilon | \overline{B}^j_{2 r_{\varepsilon}} \setminus B^j_{ r_{\varepsilon}}} \big( r_{\varepsilon} x \big) \|_{C^{4,\alpha} (\overline{B}_2 \setminus B_{1})} \nonumber \\
\leq & c_{\tau} \varepsilon^{\theta} \sum_{j} \bigg( \| \upsilon_j - \widetilde{\upsilon}_j \|_{C^{4,\alpha}} + \| \varsigma_j - \widetilde{\varsigma}_j \|_{C^{2,\alpha}} \bigg),
\end{align}
where $c_{\tau}$ depends only on $\tau$.
\end{prop}

The proof of Proposition \ref{exteriorsoln} is via the Contraction Mapping Theorem. The idea is to use an extension operator to rewrite the equation as a fixed-point problem on the punctured manifold, and there apply the results of Section \ref{lintheory} to show that for sufficiently small boundary data, the operator that we are seeking a fixed point of indeed has a solution, provided $\varepsilon>0$ is sufficiently small.

We want to solve the equation 
\begin{align*} S\bigg(\omega + \ddb \big( \widetilde{\phi} + \widetilde{\lambda}t \big) \bigg) = \widetilde{h} + \varrho(\widetilde{f} ) + \frac{1}{2} \langle \nabla \big( \widetilde{h} + \varrho(\widetilde{f} ) \big), \nabla \big( \widetilde{\phi}  + \widetilde{\lambda}t \big) \rangle,
\end{align*}
for a function $\widetilde{\varphi}$, holomorphy potential $\widetilde{h}$, constant $\widetilde{\lambda}$ and zero average holomorphy potential $\widetilde{f}$ on $D$.  We want to recast this as a perturbation problem, using the approximate solutions of the previous sections. Using the functions $\Gamma, f_{\Gamma}, h_{\Gamma}$ and constant $\lambda_{\Gamma}$ of equation \eqref{gammadefn}, as well as $W_{\varepsilon}$ of \eqref{Wepsdefn} corresponding to our choice of $a_j, \upsilon_j$ and $\varsigma_j$, we expand 
\begin{align*} \widetilde{\phi} =& \varepsilon^{2n-2} \Gamma  + W_{\varepsilon} + \phi ,\\
 \widetilde{h} =& S(\omega)  + \varepsilon^{2n-2} h_{\Gamma} + h , \\
\widetilde{\lambda} =&\varepsilon^{2n-2} \lambda_{\Gamma}  + \lambda \\
\widetilde{f} =& \varepsilon^{2n-2} f_{\Gamma} + f.
\end{align*}
As we will see in Proposition \ref{finalestimatesprop}, $(\phi, h, \lambda, f) = (0,0,0,0)$ then gives a good approximate solution to the extremal equation.

The equation we wish to solve for $\phi$ and $h$ can then be written
\begin{gather*} S \bigg(\omega + \ddb \big( \varepsilon^{2n-2} ( \Gamma + \lambda_{\gamma}t ) + W_{\varepsilon} + \phi \big) \bigg) \\ 
= S(\omega) + \varepsilon^{2n-2} \big( h_{\Gamma} + \varrho (f_{\Gamma} ) \big) + h + \varrho (f_{\gamma}) \\
 + \frac{1}{2} \langle \nabla \big( S(\omega) + \varepsilon^{2n-2} ( h_{\Gamma}+\varrho (f_{\gamma}) ) + h + \varrho(f) \big), \nabla \big( \varepsilon^{2n-2} (\Gamma + \lambda_{\Gamma}t) + W_{\varepsilon} + \phi + \lambda t \big) \rangle 
\end{gather*}

Let $L$ be the linearisation of $S\big( \omega + \ddb ( \cdot ) \big)$ at $0$. Then $L$ is given by $$L(\cdot) = - \mathcal{D}^*_{\omega} \mathcal{D}_{\omega} + \frac{1}{2} \langle \nabla (S(\omega)), \nabla (\cdot) \rangle.$$ We also have an expansion
\begin{align*} S\big( \omega + \ddb ( \cdot ) \big) &= S(\omega) + L ( \cdot ) + Q (\cdot),
\end{align*}
for some non-linear operator $Q$. In particular, we can rewrite the above equation as
\begin{align*} &\mathcal{D}^*_{\omega} \mathcal{D}_{\omega} \big( \phi + \lambda t\big) + h + \varrho(f) \\
=& Q \big( \varepsilon^{2n-2} ( \Gamma + \lambda_{\gamma}t ) + W_{\varepsilon} + \phi + \lambda t\big)  - \varepsilon^{2n-2} \big( h_{\Gamma} + \varrho ( f_{\gamma} ) \big)  \\
& - \mathcal{D}^*_{\omega} \mathcal{D}_{\omega} \big(\varepsilon^{2n-2} (\Gamma+\lambda_{\gamma}t) + W_{\varepsilon}  \big) \\
 &- \frac{1}{2} \langle \nabla \big( \varepsilon^{2n-2} ( h_{\Gamma} + \varrho (f_{\gamma}) )  + h + \varrho (f) \big), \nabla \big( \varepsilon^{2n-2} ( \Gamma + \lambda_{\gamma}t) + W_{\varepsilon} + \phi + \lambda t \big) \rangle \\
=& Q \big(  \varepsilon^{2n-2} ( \Gamma + \lambda_{\gamma}t)  + W_{\varepsilon} + \phi + \lambda t\big)  - \mathcal{D}^*_{\omega} \mathcal{D}_{\omega} \big( W_{\varepsilon}  \big) \\
&- \frac{1}{2} \langle \nabla \big( \varepsilon^{2n-2} ( h_{\Gamma} + \varrho (f_{\Gamma} ) ) + h + \varrho (f) \big), \nabla \big( \varepsilon^{2n-2} (\Gamma+\lambda_{\Gamma}t) + W_{\varepsilon} + \phi +\lambda t \big) \rangle,
\end{align*}
using that $\mathcal{D}^* \mathcal{D} \big( \Gamma + \lambda_{\Gamma} \big)  = h_{\Gamma} + \varrho(f_{\Gamma})$ away from the blown-up points. We will let $Q_{\varepsilon}$ denote the right hand side of this equation, i.e. the operator 
\begin{align*} &Q \big(  \varepsilon^{2n-2} ( \Gamma + \lambda_{\gamma}t)  + W_{\varepsilon} + \phi + \lambda t\big)  - \mathcal{D}^*_{\omega} \mathcal{D}_{\omega} \big( W_{\varepsilon}  \big) \\
&- \frac{1}{2} \langle \nabla \big( \varepsilon^{2n-2} ( h_{\Gamma} + \varrho (f_{\Gamma} ) ) + h + \varrho (f) \big), \nabla \big( \varepsilon^{2n-2} (\Gamma+\lambda_{\Gamma}t) + W_{\varepsilon} + \phi +\lambda t \big) \rangle
\end{align*}

By Proposition \ref{dwtdexplicit}, the operator $$C_{\delta, \eta}^{4,\alpha} (Y) \times \mathbb{R} \times \overline{\mathfrak{h}} \times \mathfrak{h}_D \rightarrow \widetilde{C^{0,\alpha}_{\delta-4, \eta}} (Y) $$ given by $$ (\phi, \lambda, h, f) \mapsto \mathcal{D}^* \mathcal{D}\big( \phi + \lambda t \big) + h + \varrho (f)$$ has a right inverse $P$ when $\delta \in (4-2n, 0)$ and $\eta \in (- \kappa, 0)$. Note that we are here using the decomposition of $\overline{\mathfrak{h}}_D$ as $\mathbb{R} \times \mathfrak{h}_D$ like in Lemma \ref{modifiedkercokerlem}. If our functions were defined everywhere except the blown-up points, we could then apply $P$ to our original equation to recast it as a fixed point problem. However, since our functions are only defined on the complement $Y_{\varepsilon}$ of balls around the blown-up points, we cannot do this directly. Following Arezzo-Pacard-Singer, the remedy for this is to define an extension operator $$\mathcal{E} = \mathcal{E}_{\varepsilon}: C^{4,\alpha}_{\delta, \eta} (Y_{\varepsilon}) \rightarrow C^{4,\alpha}_{\delta, \eta} (Y),$$ and apply this to the equation before applying $P$. 

The extension operator is defined as follows. At a scale $r$ it is defined to be 
\begin{itemize} \item $\mathcal{E} (f) =f$ outside $\cup_{j=1}^k \overline{B}^j_r$ 
 \item $\mathcal{E} (f) (z^j) = \frac{2|z^j|-r}{r} f(r \frac{z^j}{|z^j|})$ in $\overline{B}^j_r \setminus B^j_{r/2}$
\item $\mathcal{E}(f) = 0$ in each $\overline{B}^j_{r/2}$.
\end{itemize}
We will let $\mathcal{E}_{\varepsilon}$ denote the above operator on the scale $r_{\varepsilon}$. A key property for us is that, independently of $\varepsilon$, $\mathcal{E}_{\varepsilon}$ is a bounded operator $C^{4,\alpha}_{\delta, \eta} (Y_{\varepsilon}) \rightarrow C^{4,\alpha}_{\delta, \eta} (Y)$.

Using the extension operator, we can then rewrite the equation as a fixed point problem
\begin{align}\label{fixedptproblem}  (\phi, \lambda, h, f) = \mathcal{N}_{\varepsilon} (\phi, \lambda, h, f)  ,
\end{align}
where $\mathcal{N}_{\varepsilon}$ is the operator
\begin{align}\label{fixedptop} \mathcal{N}_{\varepsilon}= P \circ \mathcal{E}_{\varepsilon} \circ Q_{\varepsilon}.
\end{align}
A solution to the fixed point problem will then give a solution to the extremal equation on $Y_{\varepsilon}$. Note that $\mathcal{Q}_{\varepsilon}$, and therefore also $\mathcal{N}_{\varepsilon}$, depends on the boundary data $\upsilon_j, \varsigma_j$.

The fixed point is guaranteed by the Contraction Mapping Theorem once the following Proposition is proved. Recall that the boundary data $\underline{\upsilon} = (\upsilon_1, \cdots, \upsilon_k) \in \big( C^{4,\alpha} ( \partial B_1) \big)^k$ and $\underline{\varsigma} = (\varsigma_1, \cdots, \varsigma_k) \in \big( C^{2,\alpha} ( \partial B_1) \big)^k$ is assumed to satisfy the estimate \eqref{boundarycondout} and condition \eqref{kcond}. Below we will let $\mathcal{N} = \mathcal{N}_{\varepsilon}$ denote the operator \eqref{fixedptop} associated to this boundary data. We will also let $\xi= (\phi, \lambda, h, f) $ and similarly for $\xi'$. 

\begin{prop}\label{finalestimatesprop} For each $\tau > 0$ there is a $c_{\tau} > 0$ and $\varepsilon_{\tau}>0$ such that for all $\varepsilon \in (0,\varepsilon_{\tau})$, 
\begin{align}\label{approxsolestimate} \| \mathcal{N} (0,0,0,0) \| \leq c_{\tau} \big( r_{\varepsilon}^{2n+1} + \varepsilon^{4n-4} r_{\varepsilon}^{6-4n-\delta}\big)
\end{align}
and
\begin{align}\label{changeinputestimate} \| \mathcal{N} (\xi) - \mathcal{N} (\xi') \| \leq  c_{\tau} \varepsilon^{2n-2} r_{\varepsilon}^{6-4n-\delta} \| \xi - \xi'\|,
\end{align}
provided $\xi, \xi'$ have norm at most $2c_{\tau}( r_{\varepsilon}^{2n+1} + \varepsilon^{4n-4} r_{\varepsilon}^{6-4n-\delta})$.

Moreover, if $\widetilde{\mathcal{N}}$ is the map associated to a different choice of boundary data $\widetilde{\upsilon}_j, \widetilde{\varsigma}_j$ also satisfying \eqref{boundarycondout} and  \eqref{kcond} then
\begin{align}\label{changeinitialestimate} \| \mathcal{N}(\xi) - \widetilde{\mathcal{N}} (\xi) \| \leq c_{\tau} \big( r_{\varepsilon}^{2n-3} + \varepsilon^{2n-2} r_{\varepsilon}^{2-2n-\delta}\big) \| (\underline{\upsilon} - \widetilde{\underline{\upsilon}},\underline{\varsigma} - \widetilde{\underline{\varsigma}}) \|
\end{align}
for all $\xi$ satisfying $\| \xi \| \leq 2 c_{\tau}\big( r_{\varepsilon}^{2n+1} + \varepsilon^{4n-4} r_{\varepsilon}^{6 - 4n - \delta} \big).$ 
\end{prop}
In the above Proposition, the norm on the right hand side of \eqref{changeinitialestimate} is the product norm on $\big( C^{4, \alpha} (\partial B_1) \big)^{k}  \times \big( C^{2, \alpha}(\partial B_1) \big)^{k}$. 

This Proposition allows us to use the Contraction Mapping Theorem, because the estimate \eqref{changeinputestimate} shows that $\mathcal{N}$ is a contraction on the set 
\begin{align} \{ \xi \in C^{4,\alpha}_{\delta, \eta} (Y_{\varepsilon} ) \times \mathbb{R} \times \overline{\mathfrak{h}} \times \mathfrak{h}_D : \| \xi \| \leq 2 c_{\tau} \big( r_{\varepsilon}^{2n+1} + \varepsilon^{4n-4} r_{\varepsilon}^{6-4n-\delta} \big) \} ,
\end{align}
provided we (potentially) reduce $\varepsilon_{\tau}$ such that $$ c_{\tau} \leq \frac{1}{2} \varepsilon^{2n-2} r_{\varepsilon}^{6-4n-\delta}$$ for all $\varepsilon \in (0, \varepsilon_{\tau})$. Moreover, \eqref{approxsolestimate} shows that the origin is in this set. Finally, the estimate \eqref{changeinitialestimate} then shows that the metrics constructed when applying the Contraction Mapping Theorem satisfy the estimate \eqref{outvariationbnd}.

We will now prove Proposition \ref{finalestimatesprop}, which, by the above argument, completes the proof of Proposition \ref{exteriorsoln}. We follow very closely the argument of \cite{arezzopacardsinger11}, with some input from \cite{szekelyhidi12}.
\begin{proof} The right inverse $P$ is bounded independently of $\varepsilon$. By the boundedness of $\mathcal{E}$, it therefore suffices to establish a corresponding bound for $\mathcal{Q}_{\varepsilon} (0,0,0,0)$ to show that \eqref{approxsolestimate} holds. Note that 
\begin{align*} Q_{\varepsilon} (0,0,0, 0) =& Q \big( \varepsilon^{2n-2} ( \Gamma + \lambda_{\Gamma} t) + W_{\varepsilon} \big)  - \mathcal{D}^*_{\omega} \mathcal{D}_{\omega} \big( W_{\varepsilon}  \big)  \\
&- \frac{1}{2} \langle \nabla \big( \varepsilon^{2n-2}( h_{\Gamma}  + \varrho (f_{\Gamma} )) \big), \nabla \big( \varepsilon^{2n-2} ( \Gamma + \lambda_{\Gamma}t) + W_{\varepsilon} \big) \rangle.
\end{align*}
The latter of these terms satisfies the required bound because of linearity and so we can take the $\varepsilon$-dependent terms out as a factor (here we are using that $\nabla h_{\Gamma}$ vanishes at each blow-up point to get a sufficiently good bound). The bound on the second term follows because $$ \mathcal{D}^*_{\omega} \mathcal{D}_{\omega} \big( W_{\varepsilon}  \big) = \mathcal{D}^*_{\omega} \mathcal{D}_{\omega} \big( W_{\varepsilon}  \big) - \Delta^2 (W_{\varepsilon})$$ in the complement of $ B^j_{\varepsilon}$ in the ball of radius $1$ about $p_j$. Here $\Delta$ is the \textit{Euclidean} Laplacian. Since we are in normal coordinates, and the leading order term of $ \mathcal{D}^*_{\omega} \mathcal{D}_{\omega}$ equals $\Delta^2$, this implies the bound we require on the middle term.

We are left with estimating $Q(\varepsilon^{2n-2} (\Gamma + \lambda_{\Gamma}t ) + W_{\varepsilon}),$ and the key is to obtain an estimate near the blow-up points. To establish this bound near these points, we use the fact that for any subset $U$ of the blow-up of $X$, and negative $\delta$, there is a $c>0$ such that if some function $v$ is sufficiently small in $C^{4,\alpha}_{2} (U)$, then 
$$\| Q (v) \|_{C^{0,\alpha}_{\delta-4}(U)} \leq c \| v \|_{C^{4,\alpha}_{2} (U)} \| v \|_{C^{4,\alpha}_{\delta} (U)},$$ see \cite[Proposition 25]{szekelyhidi12}. Note that since we are applying this to a subset $U$ close the the blow-up points, we can assume that $D$ is far away from $U$, and so we are considering only the blow-up weights here, and can ignore the divisor weight $\eta$. Using the $\varepsilon$-dependence of the functions we are applying this to, we get precisely the required bound, as in the compact case.

For the second estimate, \eqref{changeinputestimate}, the boundedness of $P$ and $\mathcal{E}$ together with the Mean Value Theorem implies that it suffices to establish the bound for the linearised operator of $\mathcal{Q}_{\varepsilon}$ at a convex combination $ \varphi $ of the two functions. But the linearised operator of $\mathcal{Q}_{\varepsilon}$ at $\varphi$ equals $L_{\omega_{\varphi}} - L_{\omega}$, see \cite[Lemma 21]{szekelyhidi12}. 

Near the blown up points, this bound is similar to the bound on $Q(v)$ above. The key fact is that for any subset $U$ of the blow-up of $X$, and negative $\delta$, there is a $c>0$ such that if $\varphi$ is sufficiently small in $C^{4,\alpha}_{2} (U)$, then 
$$\| L_{\omega_{\varphi}} (v) - L_{\omega} (v) \|_{C^{0,\alpha}_{\delta-4}(U)} \leq c \| \varphi \|_{C^{4,\alpha}_{2} (U)} \| v \|_{C^{4,\alpha}_{\delta} (U)},$$ see \cite[Proposition 20]{szekelyhidi12}. This gives us exactly the required bound. Indeed, $\varphi$ is a convex combination of $\phi$ and $\phi'$, and so $\| \varphi \|_{C^{4,\alpha}_{2} (U)}$ will be bounded above by $\| \phi \|_{C^{4,\alpha}_{2} (U)}  + \| \phi' \|_{C^{4,\alpha}_{2} (U)}$. But $\| \phi \|_{C^{4,\alpha}_{2} (U)} \leq \varepsilon^{\delta-2} \| \phi \|_{C^{4,\alpha}_{\delta} (U)},$ and similarly for $\phi'$, using the comparison of weights (see e.g. \cite[p. 167]{szekelyhidi14book}). Combining this with the assumption on the $\delta$-norm of $\phi$ and $\phi'$ then gives the required inequality, after possibly reducing $\varepsilon_{\tau}$, by using that $\delta > 4 - 2n$. Near the divisor, the argument works in the same way: the two operators agree with the model one to highest order, which allows us to obtain a similar bound using the same strategy. 

Finally, the third estimate is also obtained using an analogous strategy. The key is to use the change in the functions like $W_{\varepsilon}$ that are associated to the boundary data $\upsilon_j, \varsigma_j$ in the corresponding estimates. For details, we refer to \cite{arezzopacardsinger11} and the earlier works \cite{arezzopacard06} and \cite{arezzopacard09}.
\end{proof}

\subsection{Matching the metrics} 

Following \cite{arezzopacardsinger11}, in order to see that we can match up the metrics created in the previous section, we will near where the regions $Y_{\varepsilon}$ and $\cup_{j=1}^k \Blball$ pull the potentials back to some fixed annular region $B_2 \setminus B_{\frac{1}{2}}$.

The system we need to solve is the following.
\begin{prop}[{\cite[Section 7]{arezzopacardsinger11}}]\label{matchingprop} Suppose that $\phi_j \in C^{4,\alpha} \big( B_2 \setminus B_1 \big)$ and $\psi_j \in C^{4,\alpha} \big( B_1 \setminus B_{\frac{1}{2}} \big)$ are the functions obtained via the $\varepsilon$-dependent charts from Proposition \ref{exteriorsoln} and \ref{interiorsoln}, respectively, for the same vector field in $\mathfrak{t}$, in such a way that on $\partial B_1$
\begin{align*} \psi_j =& \phi_j ,\\
\partial_r \psi_j =& \partial_r \phi_j ,\\
\Delta \psi_j =& \Delta \phi_j ,\\
\partial_r  \Delta \psi_j =& \partial_r \Delta \phi_j .
\end{align*} 
Then $\phi_j$ and $\psi_j$ glue across $\partial B_1$ to produce a smooth function on $\BlX \setminus D$ inducing a Poincar\'e type metric in the class $\Omega_{\varepsilon}$ given by equation \eqref{epsilonclass}, which is extremal provided the assumption \eqref{extremalvf} holds. 
\end{prop}
Note that it is because of the requirement that the vector fields above are the same that we must insist that $h_{\Gamma} \in \overline{\mathfrak{t}}$ in equation \eqref{gammadefn}.

The proof that this indeed is sufficient is exactly as in \cite{arezzopacardsinger11}, with one extra point to take care of. Their argument shows that establishing the above allows us to construct a potential of class $\Omega_{\varepsilon}$ for a Poincar\'e type metric on $\BlX \setminus D$. We then need to show that this metric is in fact extremal. In the Arezzo-Pacard-Singer setting this is automatic, but we have the possibility that we used some of the additional cokernel elements coming from pulled back functions from the divisor, and so it may be that the metric constructed is not extremal in some region away from the blow-up points.

We remove this possibility by using our assumption \eqref{extremalvf}, together with the following Lemma.
\begin{lem}\label{projectionchange} Let $X$ be a compact complex manifold, $D \subseteq X$ a smooth divisor, and $\omega$ a metric of Poincar\'e type. Let $\mathfrak{X}$ be the vector field obtained by first projecting $S(\omega)$ to $\overline{\mathfrak{h}}$ and then taking the gradient. Then the vector field is unchanged if we replace $\omega$ by $\omega_{\phi} = \omega + \ddb \phi$ with $\phi \in C^{4,\alpha}_{\eta}.$
\end{lem}
The above Lemma implies that the assumption \eqref{extremalvf} on the extremal vector field of the class $\Omega_{\varepsilon}$ ensures that projection of the scalar curvature of any metric of the type we construct in $\Omega_{\varepsilon}$ has 0 component coming from the pulled back functions from $D$, i.e. the component $\varrho(f)$ is actually $0$. This is because the approximate metric built from the Burns-Simanca metric has associated vector field which restricts to the extremal vector field of $D$ on $D$, and if two vector fields in $\mathfrak{h}$ have the same restriction to $D$ and same projection to $\mathfrak{h}_{//}^D$, then they are equal. Therefore the metrics constructed by the Cauchy matching technique have scalar curvature that actually lies in $\overline{\mathfrak{t}}$, and so the metrics are extremal.

We now prove the Lemma.
\begin{proof} Let $f_1, \cdots, f_k$ be a basis for $\overline{\mathfrak{h}}_{\omega}$. Then a basis for $\overline{\mathfrak{h}}_{\omega_{\phi}}$ is given by replacing $f_j$ by the function $$f_j^{\phi} = f_j + \frac{1}{2} \langle \nabla f_j , \nabla \phi \rangle.$$ The projection map is therefore \begin{align}\label{projectionop} \phi \mapsto \sum_{j=1}^k \big( \int_X S(\omega_{\phi}) f_j^{\phi} \omega_{\phi}^n \big) f_j^{\phi}.\end{align}

We think of this as a map into $\mathbb{R}^k$ using our chosen $\phi$-dependent bases. It suffices to show that the derivative of this map at any point is $0$. 

Let $L$ denote the derivative of the scalar curvature operator at $\phi = 0$. Recall also that the derivative of the operator $\phi \mapsto \omega_{\phi}^n$ at $\phi =0$ is $\psi \mapsto \Delta (\psi) \omega^n$, where $\Delta$ is the Laplace operator of $\omega$. The derivative at $\phi=0$ of the $j^{\textnormal{th}}$ component of the map \eqref{projectionop} is therefore the map $C^{4,\alpha}_{\eta} (X \setminus D) \rightarrow \mathbb{R}$ given by
\begin{align*} \psi \mapsto & \int_X L(\psi) f_j \omega^n + \frac{1}{2}  \int_X S(\omega)\langle \nabla f_j , \nabla \psi \rangle \omega^n + \int_X S(\omega) f_j \Delta(\psi) \omega^n .
\end{align*}
Integrating the term involving the Laplacian by parts and using that $L (\psi ) = - \mathcal{D}_{\omega}^* \mathcal{D}_{\omega} (\psi) + \frac{1}{2} \langle \nabla S(\omega), \nabla \psi \rangle$ we therefore get that the derivative is simply
\begin{align*} \psi \mapsto & - \int_X \mathcal{D}_{\omega}^* \mathcal{D}_{\omega} (\psi) f_j \omega^n .\end{align*}
But from Proposition \ref{kercokerlem} we know that the image of $\mathcal{D}_{\omega}^* \mathcal{D}_{\omega} $ on $C^{4,\alpha}_{\eta}$ is the $L^2$-orthogonal complement to $\overline{\mathfrak{h}}$. So since $f_j \in \overline{\mathfrak{h}}$, the derivative map is just $0$. Calculating the derivative of the operator at any other $\phi$ is the same, just replacing $\omega$ by $\omega_{\phi}$ above. Thus the derivative is $0$ at any $\phi$ and the Lemma is proved.
\end{proof}

The argument to show that we can actually find functions satisfying Proposition \ref{matchingprop} is exactly as in \cite[pp.39-41]{arezzopacardsinger11}. By letting the divisor volume factors $a_j$ vary for the blown-up regions we glue in, we recover the degrees freedom lost by the conditions on the boundary data $\upsilon_j, \varsigma_j$. By expanding using the low order approximations \eqref{Vepsdefn} and \eqref{Wepsdefn}, it can be shown that the matching can be achieved. This step hinges on \cite[Lemma 7.0.2]{arezzopacardsinger11}, an isomorphism result for a map between boundary data. Since this does not see the Poincar\'e type behaviour, we omit the details and refer to \cite{arezzopacardsinger11}. 

Note that this step makes us lose control of the K\"ahler class in general. However, under the assumption that any vector field in $\mathfrak{h}$ vanishing at all the points $p_i$ necessarily is in $\mathfrak{t}$, we regain this control, i.e. we can enter the K\"ahler cone in a straight line with the metrics produced, as in the compact case.

\section{Examples}\label{egsect}

In this section, we give three contexts in which we get new examples of extremal Poincar\'e type metrics using the main theorem. 

\subsection{Blowing up K\"ahler-Einstein metrics} In \cite{chengyau80}, \cite{rkobayashi84} and \cite{tianyau87} K\"ahler-Einstein Poincar\'e type metrics were constructed in the situation when $D$ is a divisor such that $K_X - D$ is ample. In this case, $X$ has no holomorphic vector fields tangent to $D$. Morever, by the adjunction formula, $K_D$ is ample in this case, and so $D$ has no holomorphic vector fields.  Therefore there are no obstructions to applying Theorem \ref{mainblthm}, and so we can blow up any finite collection of points on such manifolds, in any direction into the K\"ahler cone. The resulting metrics are then constant scalar curvature Poincar\'e type metrics. 


\subsection{Blowing up extremal toric metrics}

Another situation where Theorem \ref{mainblthm} applies is the case of toric manifolds. For these manifolds, the conditions of the theorem simplify and we begin by describing this simplification, which could be computed easily in any given example. We then give some particular cases of toric manifolds where the results apply.

Recall that a toric manifold $X$ with a K\"ahler class $\Omega$ is determined by a moment polytope $P$. Moreover, if $D$ is torus-invariant it corresponds to a facet of $F$ of the polytope. For such manifolds, all the assumptions apart from \eqref{extremalvf} in Theorem \ref{mainblthm} become redundant, by taking $T$ to be a maximal torus, which we can do provided the points we blow up are fixed points of the torus action. Thus in the compact case, one can always produce extremal metrics on the blow-up in any direction into the K\"ahler cone for such manifolds, provided one blows up fixed points of the torus action. Due to the condition \eqref{extremalvf}, this is not always sufficient in our case.

In terms of the moment polytope $P_{\varepsilon}$ of the blow-up, the condition \eqref{extremalvf} on the extremal vector field becomes that the associated affine linear function $A_{\varepsilon}$ associated to the pair $(P_{\varepsilon},F)$ differs along $F$ from the extremal affine function $A_F$ of $F$ by a constant. Here the associated affine linear function of a pair $(Q,E)$ is the unique affine linear function $A$ such that $$f \mapsto \int_{\partial Q \setminus  F} f d\sigma -  \int_{Q} f A d \lambda$$ vanishes on all affine linear function, see \cite{Don02} for details.

We begin with the case of the Abreu/Bryant metric  of \cite{abreu01} and \cite{bryant01} on $\mathbb{P}^n \setminus \mathbb{P}^{n-1}$, say with class $c_1 \big( \mathcal{O}(1) \big)$. This is an extremal Poincar\'e type metric which is not of constant scalar curvature. 

The moment polytope $P$ of $\mathbb{P}^n$ with this class is the standard simplex $$P = \{ x : x_i \geq 0, x_1 + \cdots +x_n \leq 1 \} \subseteq \mathbb{R}^n$$ and the facet $F$ corresponding to the divisor $\mathbb{P}^{n-1}$ is the boundary component $$F = \{ x  \in P : x_1 + \cdots + x_n =1 \}.$$ There is only one fixed point not on the divisor, and this corresponds to the origin in the moment polytope. Letting $y= x_1 + \cdots + x_n$, we have that the associated affine linear function $A_{(P,F)}$ is of the form $$ A_{(P,F)} = a + b y.$$

The $\varepsilon$-blow-up of $\mathbb{P}^n$ in this point has moment polytope $$P_{\varepsilon} = \{ x : x_i \geq 0, \varepsilon \leq  x_1 + \cdots +x_n \leq 1 \} .$$ Since this polytope keeps the symmetry of the function $y$, the associated affine linear function $A_{\varepsilon}$ of $(P_{\varepsilon},F)$ is of the form $$ A_{\varepsilon} = a_{\varepsilon} + b_{\varepsilon} y.$$ In particular, the restriction to $F$ is a constant, and so Theorem \ref{mainblthm} applies. In other words, there is a $\varepsilon_0 >0$ such that there is an extremal Poincar\'e type metric on the $\varepsilon$-blow-up for all $\varepsilon \in (0,\varepsilon_0).$

\begin{rem} In \cite{apostolovauvraysektnan17}, it was actually shown for the case of surfaces that $\varepsilon_0$ is as large as it can be, i.e. the Seshadri constant of the blow-up point.
\end{rem}

We can now try to blow up these again. When doing so, there is now more than one point to choose from. Choosing a single point will destroy the symmetry in the function $y$, and so the associated affine linear function will no longer be constant along $F$. However, the symmetry is kept if we blow up in all the new fixed points, with the $\varepsilon$-dependent K\"ahler class having equal volumes for all the new exceptional divisors. So for this choice of blow-up points and divisor volumes, we can apply Theorem \ref{mainblthm}.

We could continue doing this inductively, all the time blowing up all new fixed points with equal volume for all the exceptional divisors of the blow-up. Thus Theorem \ref{mainblthm} can be applied to \textit{some} successive blow-ups of the Abreu/Bryant extremal Poincar\'e type manifold. In particular, we get an infinite family of different complex manifolds (of different topological type), all admitting extremal Poincar\'e type metrics.

\bibliography{biblibrary}
\bibliographystyle{alpha}

\end{document}